\documentclass[12pt,reqno]{amsart}
\usepackage{amsmath,a4wide}
\usepackage{xfrac}
\usepackage{stmaryrd,mathrsfs,bm,amsthm,mathtools,yfonts,amssymb,color}
\usepackage{xcolor}

\begingroup
\newtheorem{theorem}{Theorem}[section]
\newtheorem{lemma}[theorem]{Lemma}
\newtheorem{proposition}[theorem]{Proposition}
\newtheorem{corollary}[theorem]{Corollary}
\endgroup

\theoremstyle{definition}
\newtheorem{definition}[theorem]{Definition}
\newtheorem{remark}[theorem]{Remark}
\newtheorem{ipotesi}[theorem]{Assumption}

\numberwithin{equation}{section}

\newcommand\supp{{\rm spt}}
\newcommand{\sing}{{\rm Sing}}
\newcommand{\reg}{{\rm Reg}}
\newcommand\rD{{\rm D}}
\newcommand\res{\mathop{\hbox{\vrule height 7pt width .3pt depth 0pt
\vrule height .3pt width 5pt depth 0pt}}\nolimits}
\newcommand{\im}{{\rm Im}}

\newcommand{\bT}{\mathbf{T}}
\newcommand{\bG}{\mathbf{G}}
\newcommand{\cH}{{\mathcal{H}}}
\newcommand{\bJ}{{\mathbf{J}}}

\newcommand{\p}{{\mathbf{p}}}
\newcommand{\q}{{\mathbf{q}}}

\newcommand{\sW}{{\mathscr{W}}}

\newcommand\bmo{{\bm m}_0}

\newcommand{\cM}{{\mathcal{M}}}
\newcommand{\bU}{{\mathbf{U}}}

\newcommand{\phii}{{\bm{\varphi}}}
\newcommand{\Phii}{{\bm{\Phi}}}

\newcommand{\cU}{{\mathcal{U}}}

\newcommand{\bD}{{\mathbf{D}}}
\newcommand{\bH}{{\mathbf{H}}}
\newcommand{\bI}{{\mathbf{I}}}
\newcommand{\bE}{{\mathbf{E}}}
\newcommand{\bF}{{\mathbf{F}}}
\newcommand{\bOmega}{{\mathbf{\Omega}}}
\newcommand{\bSigma}{{\mathbf{\Sigma}}}

\newcommand{\bB}{{\mathbf{B}}}

\newcommand{\be}{{\mathbf{e}}}

\newcommand\N{{\mathbb N}}

\newcommand\R{{\mathbb R}}

\newcommand{\eps}{{\varepsilon}}

\def\Xint#1{\mathchoice
{\XXint\displaystyle\textstyle{#1}}%
{\XXint\textstyle\scriptstyle{#1}}%
{\XXint\scriptstyle\scriptscriptstyle{#1}}%
{\XXint\scriptscriptstyle\scriptscriptstyle{#1}}%
\!\int}
\def\XXint#1#2#3{{\setbox0=\hbox{$#1{#2#3}{\int}$ }
\vcenter{\hbox{$#2#3$ }}\kern-.6\wd0}}
\def\mint{\Xint-}

\newcommand{\cone}{{\times\hspace{-0.6em}\times\,}}
\newcommand{\Lip}{{\rm {Lip}}}

\newcommand{\dist}{{\rm {dist}}}
\newcommand{\dv}{{\text {div}}}

\newcommand\Id{{\rm Id}\,}

\newcommand{\cA}{{\mathcal{A}}}
\newcommand{\cB}{{\mathcal{B}}}
\newcommand{\cG}{{\mathcal{G}}}

\newcommand{\cF}{{\mathcal{F}}}

\newcommand{\cR}{{\mathcal{R}}}

\newcommand\sS{{\mathscr S}}

\newcommand{\mass}{{\mathbf{M}}}

\newcommand{\Iq}{{\mathcal{A}}_Q}
\def\a#1{\left\llbracket{#1}\right\rrbracket}

\newcommand{\D}{\textup{Dir}}
\newcommand{\de}{\partial}

\newcommand{\etaa}{{\bm{\eta}}}
\newcommand{\ph}{\varphi}

\newcommand\B{{\mathbf{B}}}
\newcommand{\bh}{\mathbf{h}}
\newcommand\J{{\mathbf{J}}}

\title[Regularity of area minimizing currents III: Blow-up]{Regularity of area minimizing currents III:\\ Blow-up}
\author{Camillo De Lellis}
\address{Mathematik Institut der Universit\"at Z\"urich}
\email{delellis@math.uzh.ch}

\author{Emanuele Spadaro}
\address{Max-Planck-Institut f\"ur Mathematik in den Naturwissenschaften, Leipzig}
\email{spadaro@mis.mpg.de}

\subjclass[2010]{49Q15, 49N60, 49Q05}
\keywords{Integer rectifiable currents, Regularity, Area minimizing}

\begin{document}
\begin{abstract}
This is the last of a series of three papers in which
we give a new, shorter proof of 
a slightly improved version of
Almgren's partial regularity of area mi\-ni\-mi\-zing currents
in Riemannian manifolds.
Here we perform a blow-up analysis  deducing the regularity of area minimizing currents from that of Dir-minimizing multiple valued functions.
\end{abstract}

\maketitle

\section{Introduction}
In this paper we complete the proof of 
a slightly improved version of
the celebrated Almgren's partial regularity result for area minimizing currents in a Riemannian manifold (see \cite{Alm}),
namely Theorem \ref{t:main} below.

\begin{ipotesi}\label{ipotesi_principale}
Let $\eps_0\in ]0,1[$, $m,\bar{n}\in \N\setminus \{0\}$ and $l\in \N$. We denote by
\begin{itemize}
\item[(M)] $\Sigma\subset\R^{m+n} = \R^{m+\bar{n}+l}$ an embedded $(m+\bar{n})$-dimensional submanifold of class $C^{3,\eps_0}$;
\item[(C)] $T$ an integral current of dimension $m$ with compact support $\supp (T) \subset \Sigma$,
area minimizing in $\Sigma$.
\end{itemize}
\end{ipotesi}

In this paper we follow the notation of \cite{DS4} concerning balls, cylinders and disks. In particular $\bB_r (x)\subset \mathbb R^{m+n}$ will denote the Euclidean ball of radius $r$ and center $x$.

\begin{definition}\label{d:reg_sing}
For $T$ and $\Sigma$ as in Assumption \ref{ipotesi_principale} we define
\begin{align}
&\reg (T) := \big\{x\in \supp (T) :\; \supp(T)\cap  \bB_r (x) \mbox{ is a $C^{3,\eps_0}$ submanifold for some $r>0$}\big\},\\
&\sing (T) := \supp (T) \setminus \big(\supp (\partial T) \cup \reg (T)\big).
\end{align}
\end{definition}
The partial regularity result proven first by Almgren \cite{Alm} under the more
restrictive hypothesis $\Sigma \in C^5$
gives an estimate on the Hausdorff dimension ${\rm dim}_H (\sing (T))$ of $\sing(T)$.

\begin{theorem}\label{t:main}
${\rm dim}_H (\sing(T))\leq m-2$ for any
$m,\bar{n}, l, T$ and $\Sigma$ as in Assumption \ref{ipotesi_principale}.
\end{theorem}


In this note we complete the proof of Theorem~\ref{t:main}, based
on our previous works \cite{DS1, DS2, DS3, DS4}, thus
providing a new, and much shorter, account of one of the most fundamental
regularity result in geometric measure theory; we refer to \cite{DS3} for
an extended general introduction to all these works.
The proof is carried by contradiction: in the sequel we
will always assume the following.

\begin{ipotesi}[Contradiction]\label{assurdo}
There exist $m \geq 2, \bar{n}, l$, $\Sigma$ and $T$  as in Assumption \ref{ipotesi_principale} such that 
$\cH^{m-2+\alpha}(\sing(T)) >0$ for some $\alpha>0$. 
\end{ipotesi}

The hypothesis $m\geq 2$ in Assumption \ref{assurdo} is justified by the well-known fact that
$\sing (T) = \emptyset$ when $m=1$ (in this case $\supp (T)\setminus \supp (\partial T)$
is locally the union of finitely many non-intersecting geodesic segments).
Starting from Assumption \ref{assurdo}, we make a careful blow-up analysis, split in the following steps.

\medskip

\subsection{Flat tangent planes}
We first reduce to flat 
blow-ups around a given point, which in the sequel is assumed to be the origin.  
These blow-ups will also be chosen so that the size of the singular set satisfies a uniform estimate from below (cf.~Section~\ref{s:flat}).
 
\subsection{Intervals of flattening}
For appropriate rescalings of the current around the origin 
we take advantage of the center manifold constructed
in \cite{DS4}, which
gives a good approximation of the average of the sheets of the
current at some given scale. 
However, since it might fail to do so at different scales,
in Section~\ref{s:flattening} 
we introduce a \textit{stopping condition} for the center manifolds
and define appropriate {\em intervals of flattening $I_j = [s_j, t_j]$}.
For each $j$ we construct a different center manifold 
$\cM_j$ and approximate the (rescaled) current with 
a suitable multi-valued map on the normal bundle of $\cM_j$.

\subsection{Finite order of contact}
A major difficulty in the analysis is to prove that the minimizing current has finite order of contact with the center manifold.
To this aim, in analogy with the case of harmonic multiple valued functions
(cf.~\cite[Section 3.4]{DS1}),
we introduce a variant of the \textit{frequency function} and prove its
almost monotonicity and boundedness.
This analysis, carried in Sections~\ref{s:frequency}, ~\ref{s:variationII} and ~\ref{s:freq_bound}, relies on the variational formulas 
for images of multiple valued maps as computed in \cite{DS2} and
on the careful estimates of \cite{DS4}. 
Our frequency function differs from that of Almgren and allows for simpler estimates.

\subsection{Convergence to $\D$-minimizer and contradiction}
Based on the previous steps, we can blow-up the Lipschitz approximations
from the center manifold $\cM_j$ in 
order to get a limiting Dir-minimizing function on a flat $m$-dimensional domain.
We then show that the singularities of the rescaled currents converge to singularities of that limiting ${\rm Dir}$-minimizer, contradicting the partial regularity of \cite[Section 3.6]{DS1} and,
hence, proving Theorem \ref{t:main}.

\medskip

{\bf Acknowledgments.}
The research of Camillo De Lellis has been supported by the ERC grant
agreement RAM (Regularity for Area Minimizing currents), ERC 306247. 
The authors are warmly thankful to Bill Allard for several enlightening
conversations and his constant support.

\section{Flat tangent cones}\label{s:flat}

\begin{definition}[$Q$-points] \label{d:regular_points}
For $Q\in \N$, we denote by $\rD_Q(T)$ the points of density $Q$ of the current $T$,
and set
\begin{gather*}
\reg_Q (T) := \reg (T)\cap \rD_Q (T) \quad\text{and}\quad
\sing_Q (T) := \sing (T)\cap \rD_Q (T).
\end{gather*}
\end{definition}

\begin{definition}[Tangent cones]\label{d:t_cones}
For any $r>0$ and $x\in \R^{m+n}$, $\iota_{x,r}: \R^{m+n}\to \R^{m+n}$ is the map
$y\mapsto \frac{y-x}{r}$ and
$T_{x, r} := (\iota_{x, r})_{\sharp} T$.
The classical monotonicity formula (see \cite{Sim} and \cite[Lemma~A.1]{DS3}) implies that, for every $r_k\downarrow 0$ and $x\in \supp(T) \setminus \supp(\de T)$,
there is a subsequence (not relabeled) for which $T_{x, r_k}$ converges
to an integral cycle $S$ which is a cone (i.e., $S_{0,r} = S$ for all $r>0$ and $\partial S=0$)
and is (locally) area-minimizing in $\R^{m+n}$. Such a cone will be called, as usual, {\em a tangent cone to $T$ at $x$}.
\end{definition}

Fix $\alpha>0$. By Almgren's stratification theorem (see \cite[Theorem 35.3]{Sim}),
for $\cH^{m-2+\alpha}$-a.e. $x \in \supp(T) \setminus \supp(\de T)$,
there exists a subsequence of radii $r_k \downarrow 0$
such that $T_{x,r_k}$ converge to an integer multiplicity flat plane.
Similarly, for measure-theoretic reasons, if $T$ is as in Assumption \ref{assurdo}, 
then for $\cH^{m-2+\alpha}$-a.e.~$x \in \supp(T) \setminus \supp(\de T)$
there is a subsequence $s_k\downarrow 0$ such that
$\liminf_k \cH^{m-2+\alpha}_\infty (\rD_Q (T_{x,s_k}) \cap \B_1) >0$ (see again \cite{Sim}). Obviously
there would then be $Q\in \N$ and $x\in \sing_Q (T)$ where both subsequences
exist.
The two subsequences might, however, differ: in the next proposition  
we show the existence of one point and a single subsequence along which
\textit{both} conclusions hold. For the relevant notation (concerning, for instance, excess and height of currents) we refer to \cite{DS3,DS4}.

\begin{proposition}[Contradiction sequence]\label{p:seq}
Under Assumption \ref{assurdo}, there are $m,n, Q \geq 2$, $\Sigma$ and $T$ as in Assumption
\ref{ipotesi_principale},
reals $\alpha,\eta>0$, and a sequence $r_k\downarrow 0$ such that
$0\in \rD_Q (T)$ and the following holds:
\begin{gather}
\lim_{k\to+\infty}\bE(T_{0,r_k}, \B_{6\sqrt{m}}) = 0,\label{e:seq1}\\
\lim_{k\to+\infty} \cH^{m-2+\alpha}_\infty (\rD_Q (T_{0,r_k}) \cap \B_1) > \eta,\label{e:seq2}\\
\cH^m \big((\B_1\cap \supp (T_{0, r_k}))\setminus \rD_Q (T_{0,r_k})\big) > 0\quad \forall\; k\in\N.\label{e:seq3}
\end{gather}
\end{proposition}

The proof is based on the following lemma.

\begin{lemma}\label{l:regular cone}
Let $S$ be an $m$-dimensional area minimizing integral cone
in $\R^{m+n}$ such that $\partial S=0$,
$Q= \Theta (S, 0) \in \mathbb N$, $\cH^m\big(D_Q(S)) >0$ and $\cH^{m-1}(\sing_Q(S))=0$.
Then, $S$ is an $m$-dimensional plane with multiplicity $Q$.
\end{lemma}

\begin{proof}
For each $x\in \reg_Q (S)$, let  $r_x$ be such that $S\res \B_{2r_x} (x) = Q \a{\Gamma}$
for some regular submanifold $\Gamma$ and set
\[
U := \bigcup_{x \in \reg_Q(S)} \B_{r_x} (x).
\]
Obviously, $\reg_Q (S)\subset U$; hence, by assumption, it is not empty.
Fix $x\in \supp (S)\cap \partial U$. 
Let next $(x_k)_{k \in \N}\subset \reg_Q (S)$ be such that ${\rm dist}\, (x, \B_{r_{x_k}}(x_k)) \to 0$.
We necessarily have that $r_{x_k}\to 0$: otherwise we would have
$x\in \B_{2r_{x_k}} (x_k)$ for some $k$, which would imply $x\in \reg_Q (S)\subset U$, i.e. a contradiction.
Therefore, $x_k\to x$ and,
by \cite[Theorem~35.1]{Sim},
\[
Q=\limsup_{k\to+\infty}\Theta(S,x_k)  \leq \Theta (S, x) = \lim_{\lambda \downarrow 0} \Theta (S, \lambda x) \leq \Theta(S,0) = Q.
\]
This implies $x\in \rD_Q (S)$.
Since $x\in \partial U$, we must then have $x\in \sing_Q (S)$.
Thus, we conclude that
$\cH^{m-1} (\supp (S)\cap \de U) = 0$.
It follows from the standard theory of rectifiable currents
(cf.~Lemma~\ref{l:separate}) that
$S':=S\res U$ has $0$ boundary in $\R^{m+n}$.
Moreover, since $S$ is an area minimizing cone, $S'$ is also an area-minimizing cone. 
By definition of $U$ we have $\Theta (S', x) = Q$ for $\|S'\|$-a.e.~$x$ and, by semicontinuity,
\[
Q\leq\Theta (S', 0)\leq \Theta (S, 0) =Q.
\]
We apply Allard's theorem and deduce that $S'$ is regular,
i.e.~$S'$ is an $m$-plane with multiplicity $Q$. Finally, from $\Theta (S', 0) = \Theta (S, 0)$, we infer $\mass (S\res \B_1)=\mass (S'\res \B_1)$ and then
$S'=S$.
\end{proof}

\begin{proof}[Proof of Proposition~\ref{p:seq}] Let $m>1$
be the smallest integer for which Theorem \ref{t:main} fails.
By ~Theorem~\ref{t:strat}
there must be 
an integer rectifiable area minimizing current $R$ of dimension $m$
and a positive integer $Q$
such that the Hausdorff dimension of $\sing_Q (R)$ is larger than $m-2$ (note that Theorem \ref{t:strat} is just a corollary of a well known stratification theorem by Almgren, cf. \cite{Alm,Sim, White97}).
We fix the smallest $Q$ for which such a current $R$ exists.
Recall that, by the upper semicontinuity of the density and
a straightforward application of Allard's regularity theorem (see Theorem~\ref{t:strong_degiorgi}),
$\sing_1 (R) =\emptyset$, i.e.~$Q>1$.

Let $\alpha\in ]0,1]$ be such that $\cH^{m-2+\alpha}(\sing_Q(R))>0$. 
By \cite[Theorem~3.6]{Sim} there exists a point $x\in \sing_Q (R)$ such
that $\sing_Q (R)$ has positive $\cH^{m-2+\alpha}_\infty$-upper density: i.e., assuming without loss of generality $x=0$ and $\partial R\res \B_1 = 0$, there exists $r_k\downarrow 0$ such that
\[
\lim_{k\to+\infty} \cH^{m-2+\alpha}_\infty \big(\sing_Q (R_{0, r_k})\cap \B_1\big)
= \lim_{k\to+\infty} \frac{\cH^{m-2+\alpha}_\infty \big(\sing_Q (R)\cap \B_{r_k}\big)}{r_k^{m-2+\alpha}} > 0\, .
\]
Up to a subsequence (not relabeled) we can assume that $R_{0,r_k}\to S$, with
$S$ a tangent cone. 
If $S$ is a multiplicity $Q$ flat plane, then we set $T:=R$ and we are done: indeed, \eqref{e:seq3} is satisfied by Theorem~\ref{t:strong_degiorgi}, because $0\in\sing (R)$ and $\|R\|\geq \cH^m \res \supp (R)$.

Assume therefore that $S$ is {\em not} an $m$-dimensional plane with multiplicity $Q$. 
Taking into account the convergence of the total variations for minimizing currents \cite[Theorem~34.5]{Sim} and the upper semicontinuity of
$\cH_\infty^{m-2+\alpha}$ under the Hausdorff convergence of compact sets,
we get
\begin{equation}\label{e:upper_density}
\cH^{m-2+\alpha}_\infty \big(\rD_Q (S) \cap \bar{\B}_1\big) \geq
\liminf_{k\to+\infty} \cH^{m-2+\alpha}_\infty \big(\rD_Q (R_{0, r_k})\cap \bar{\B}_1\big) > 0.
\end{equation}
We claim that \eqref{e:upper_density} implies
\begin{equation}\label{e:persistence sing}
\cH^{m-2+\alpha}_\infty (\sing_Q (S)) > 0.
\end{equation}
Indeed, if all points of $\rD_Q (S)$ are singular, then this follows from
\eqref{e:upper_density} directly.
Otherwise, $\reg_Q(S)$ is not empty and, hence, $\cH^m (D_Q (S) \cap \B_1)>0$.
In this case we can apply Lemma~\ref{l:regular cone} and infer that,
since $S$ is not regular, then
$\cH^{m-1} (\sing_Q (S))>0$ and \eqref{e:persistence sing} holds.

We can, hence, find $x\in \sing_Q (S)\setminus \{0\}$ and $r_k\downarrow 0$ such that
\[
\lim_{k\to+\infty} \cH^{m-2+\alpha}_\infty \big(\sing_Q (S_{x, r_k})\cap \B_1\big)
= \lim_{k\to+\infty} \frac{\cH^{m-2+\alpha}_\infty \big(\sing_Q (S)\cap \B_{r_k}(x))}{r_k^{m-2+\alpha}} > 0.
\]
Up to a subsequence (not relabeled), we can assume that $S_{x, r_k}$ converges to $S_1$.
Since $S_1$ is a tangent cone to the cone $S$ at $x\neq 0$, $S_1$ splits off a line,
i.e.~$S_1 = S_2\cone \a{\R v}$, for some area minimizing cone $S_2$ in $\R^{m-1+n}$
and some $v\in \R^{m+n}$ (cf.~the arguments in \cite[Lemma 35.5]{Sim}).
Since $m$ is, by assumption, the smallest integer for which Theorem \ref{t:main} fails, $\cH^{m-3+\alpha} (\sing (S_2)) = 0$ and, hence, $\cH^{m-2+\alpha} (\sing_Q (S_1))= 0$.
On the other hand, arguing as we did for \eqref{e:upper_density}, we have
\[
\cH^{m-2+\alpha}_\infty (\rD_Q (S_1)\cap \bar\B_1)\geq
\limsup_{k\to+\infty} \cH^{m-2+\alpha}_\infty (\rD_Q (S_{x, r_k})\cap \bar\B_1)  > 0.
\]
Thus $\reg_Q (S_1)\neq \emptyset$ and, hence, $\cH^m (D_Q (S_1))> 0$.
We can apply Lemma~\ref{l:regular cone} again and conclude that $S_1$ is an $m$-dimensional plane with multiplicity $Q$. 
Therefore, the proposition follows taking $T:= \tau_\sharp S$, with
$\tau$ the translation map $y\mapsto y-x$, and $\Sigma$
the tangent plane at $0$ to the original Riemannian manifold.
\end{proof}

\section{Intervals of flattening}\label{s:flattening}
For the sequel we fix the constant $c_s := \frac{1}{64\sqrt{m}}$ and notice that
$2^{-N_0} < c_s$,
where $N_0$ is the parameter introduced in \cite[Assumption 1.8]{DS4}.
It is always understood that
the parameters $\beta_2, \delta_2, \gamma_2, \kappa, C_e, C_h, M_0, N_0$
in \cite{DS4} are fixed in such a way that all the theorems and propositions
therein are applicable, cf.~\cite[Section 1.2]{DS4}. In particular, all constants which will depend upon
these parameters will be called {\em geometric} and denoted by $C_0$. On the contrary, we will highlight
the dependence of the constants upon the parameters introduced in this paper $p_1, p_2, \ldots$ by
writing $C= C(p_1, p_2, \ldots)$.

We recall also the notation introduced in \cite[Assumption 1.3]{DS4}. If $\Sigma \cap \bB_{7\sqrt{m}}$ has no boundary in $\bB_{7\sqrt{m}}$ and for any $p\in \Sigma \cap \bB_{7\sqrt{m}}$ there is a map $\Psi_p:T_p\Sigma\supset \Omega \to (T_p\Sigma)^\perp$ parametrizing it, then
$\mathbf{c} (\Sigma \cap \bB_{7\sqrt{m}}):= \sup_{p\in \Sigma\cap \bB_{7\sqrt{m}}} \|D\Psi_p\|_{C^{2,\eps_0}}$.  
Obviously these assumptions might fail for a general $\Sigma$ (in fact $\mathbf{c} (\Sigma\cap \bB_{7\sqrt{m}})$ need not be well-defined). However, having fixed a point $q\in \Sigma$, given its $C^{3,\eps_0}$ regularity, $\mathbf{c} (\iota_{q,r} (\Sigma)\cap \bB_{7\sqrt{m}})$ is well-defined whenever $r$ is sufficiently small and converges to $0$ as $r\downarrow 0$.
In particular, by Proposition~\ref{p:seq} and  simple rescaling arguments,
we assume in the sequel the following.

\begin{ipotesi}\label{i:H'} Let $\eps_3\in ]0, \eps_2[$. Under Assumption~\ref{assurdo}, there exist  $m, n, Q \geq 2$, $\alpha, \eta > 0$,  $T$ and $\Sigma$ for which:
\begin{itemize}
\item[(a)] there is a sequence of radii 
$r_k\downarrow 0$ as in Proposition \ref{p:seq};
\item[(b)] the following holds:
\begin{gather}
T_0\Sigma = \R^{m+\bar{n}} \times \{0\}, \quad 
\supp (\partial T)\cap \B_{6\sqrt{m}} = \emptyset,\quad 0 \in D_Q(T),\label{e:aggiuntive}\\
\|T\| (\B_{6\sqrt{m} r}) \leq
r^m \left( Q\,\omega_m (6\sqrt{m})^m + \eps_3^2\right) \quad \text{for all } r\in (0,1),\label{e:(2.1)-del-cm}\\
 \mathbf{c} (\Sigma\cap \bB_{7\sqrt{m}}) \leq \eps_3\, .\label{e:piattezza_sigma}
\end{gather}
\end{itemize}
\end{ipotesi}

\subsection{Defining procedure}\label{ss:flattening}
We set
\begin{equation}\label{e:def_R}
\mathcal{R}:=\big\{r\in ]0,1]: \bE (T, \B_{6\sqrt{m} r}) \leq \eps_3^2\big\}\,.
\end{equation}
Observe that, if $\{s_k\}\subset \mathcal{R}$
and $s_k\uparrow s$, then $s\in \mathcal{R}$.
We cover $\cR$ with a collection $\mathcal{F}=\{I_j\}_j$ of intervals
$I_j = ]s_j, t_j]$ defined as follows.
$t_0:= \max \{t: t\in \mathcal{R}\}$. Next assume, by induction, to have defined $t_j$ (and hence also
$t_0 > s_0\geq t_1 > s_1 \geq \ldots > s_{j-1}\geq t_j$) and consider the following objects:
\begin{itemize}
\item[-] $T_j := ((\iota_{0,t_{j}})_\sharp T)\res \bB_{6\sqrt{m}}$, $\Sigma_j := \iota_{0, t_j} (\Sigma) \cap \bB_{7\sqrt{m}}$; moreover, consider
for each $j$ an orthonormal system of coordinates so that, if we denote by $\pi_0$ the $m$-plane $\mathbb R^m\times \{0\}$, then $\bE (T_j, \bB_{6\sqrt{m}}, \pi_0) = \bE(T_j,\B_{6\sqrt{m}})$ (alternatively we can keep the system of coordinates fixed
and rotate the currents $T_j$).
\item[-] Let $\cM_j$ be the corresponding center manifold constructed
in \cite[Theorem~1.17]{DS4} applied to $T_j$ and $\Sigma_j $ with respect to the $m$-plane $\pi_0$;
the manifold $\cM_j$ is then the graph of a map $\phii_j: \pi_0 \supset [-4,4]^m \to \pi_0^\perp$,
and we set $\Phii_j (x) := (x, \phii_j (x)) \in \pi_0\times \pi_0^\perp$.
\end{itemize}
{ Then, we consider the Whitney decomposition $\sW^{(j)}$ of $[-4,4]^m \subset \pi_0$ as in \cite[Definition 1.10 \& Proposition~1.11]{DS4} (applied to $T_j$) and we define
\begin{equation}\label{e:s_j}
s_j := t_j\, \max\, \left(\{c_s^{-1} \ell (L) : L\in \sW^{(j)} \mbox{ and } c_s^{-1} \ell (L) \geq \dist (0, L)\} \cup \{0\} \right)\, .
\end{equation} 
We will prove below that $s_j/t_j <2^{-5}$. In particular this ensures that $[s_j, t_j]$ is a (nontrivial) interval. 
Next, if $s_j =0$ we stop the induction. Otherwise we let $t_{j+1}$ be the largest element
in $\mathcal{R}\cap ]0, s_j]$ and proceed as above. Note moreover the following simple consequence of \eqref{e:s_j}:
\begin{itemize}
\item[(Stop)] If $s_j >0$ and $\bar{r} := s_j/t_j$, then there is $L\in \sW^{(j)}$ with 
\begin{equation}\label{e:st}
\ell(L) = c_s\,\bar r\, \qquad \mbox{and} \qquad L\cap \bar B_{\bar r} (0, \pi_0)\neq \emptyset\, 
\end{equation}
(in what follows $B_r (p, \pi)$ and $\bar B_r (p,\pi)$ will denote the open and closed disks 
$\bB_r (p)\cap (p+\pi)$, $\bar{\bB}_r (p) \cap (p+\pi)$);
\item[(Go)] If $\rho > \bar{r} := s_j/t_j$, then
\begin{equation}\label{e:go}
\ell (L) < c_s \rho \qquad \mbox{for all } L\in \sW^{j(k)} \mbox{ with $L\cap B_\rho (0, \pi_0) \neq \emptyset$.}
\end{equation}
In particular the latter inequality is true for every $\rho\in ]0,3]$ if $s_j =0$.
\end{itemize}}

\subsection{First consequences}
The following is a list of easy consequences of the definition. Given two sets $A$ and $B$, we define their
{\em separation} as the number ${\rm sep} (A,B) := \inf \{|x-y|:x\in A, y\in B\}$.

\begin{proposition}\label{p:flattening}
Assuming $\eps_3$ sufficiently small, then the following holds:
\begin{itemize}
\item[(i)] $s_j < \frac{t_j}{2^5}$ and the family $\mathcal{F}$ is either countable
and $t_j\downarrow 0$, or finite and $I_j = ]0, t_j]$ for the largest $j$;
\item[(ii)] the union of the intervals of $\cF$ cover $\cR$,
and for $k$ large enough the radii $r_k$ in Assumption \ref{i:H'} belong to $\cR$;
\item[(iii)] if $r \in ]\frac{s_j}{t_j},3[$ and $J\in \mathscr{W}^{(j)}_n$ intersects
$B:= \p_{\pi_0} (\cB_r (p_j))$, with $p_j := \Phii_j(0)$,
then $J$ is in the
domain of influence $\mathscr{W}_n^{(j)} (H)$ (see \cite[Definition~3.3]{DS4})
of a cube $H\in \mathscr{W}^{(j)}_e$ with
\[
\ell (H)\leq 3 \, c_s\, r \quad \text{and}\quad 
\max\left\{{\rm sep}\, (H, B),  {\rm sep}\, (H, J)\right\}
\leq 3\sqrt{m}\, \ell (H) \leq \frac{3 r}{16};
\]
\item[(iv)] $\bE (T_j, \B_r )\leq C_0 \eps_3^2 \, r^{2-2\delta_2}$ for
every $r\in]\frac{s_j}{t_j},3[$.
\item[(v)] $\sup \{ \dist (x,\cM_j): x\in \supp(T_j) \cap \p^{-1}_j(\cB_r(p_j))\} \leq C_0\, (\bmo^j)^\frac{1}{2m} r^{1+\beta_2}$ for
every $r\in]\frac{s_j}{t_j},3[$, where
$\bmo^j := \max\{\mathbf{c}(\Sigma_j)^2 , \bE(T_j, \bB_{6\sqrt{m}})\}$. 
\end{itemize}
\end{proposition}

\begin{proof}
We start by noticing that $s_j\leq \frac{t_j}{2^5}$ follows from the inequality $2^{-N_0} < c_s$ (cf. \cite[Assumption 1.8]{DS4})
because all cubes in the
Whitney decomposition have side-length at most $2^{-N_0-6}$ (cf. \cite[Proposition 1.11]{DS4}).
In particular, this implies that the inductive procedure either never stops,
leading to $t_j\downarrow 0$, or it stops because $s_j =0$ and
$]0, t_j]\subset \mathcal{R}$, thus proving (i).
The first part of (ii) follows straightforwardly from the choice of $t_{j+1}$, and
the last assertion holds from $\bE(T, \B_{6\sqrt{m}r_k}) \to 0$.

Regarding (iii), let $H\in \sW^{(j)}_e$ be as in \cite[Definition 3.3]{DS4} and choose $k\in \N\setminus \{0\}$
such that $\ell(H) = 2^k\, \ell(J)$. Observe that
$\|D \phii_j\|_{C^{2,\kappa}}\leq C_0 \eps_3$ by \cite[Theorem~1.17]{DS4}. If $\eps_3$ is sufficiently small,
we can assume 
\begin{equation}\label{e:comparable balls}
B_{r/2} (0, \pi_0) \subset B \subset B_{r} (0, \pi_0)\, .
\end{equation}
Now, by \cite[Corollary 3.2]{DS4}, ${\rm sep} (H, J) \leq 2 \sqrt{m} \ell (H)$ and 
\[
{\rm sep}\, (B, H) \leq {\rm sep}\, (H, J) + 2\sqrt{m} \ell (J) \leq 3 \sqrt{m}\, \ell (H).
\]
Both the inequalities claimed in  (iii) are then trivial when
$r > \frac{1}{4}$, because $\ell (H) \leq 2^{-N_0-6} \leq 2^{-5} c_s \leq 2^{-9}/\sqrt{m}$. Assume therefore $r\leq \frac{1}{4}$ and
note that $H$ intersects $B_{2r + 3\sqrt{m}\, \ell (H)}$. Let $\rho:= 2r + 3\sqrt{m} \ell (H)$. 
Observe that $2r < \rho<1$. By the definition of $s_j$, we have that
\[
\ell (H) < c_s \, \left( 2r + 3\sqrt{m}\, \ell (H) \right)
= 2 \,c_s\, r + \frac{3 \ell (H)}{16} . 
\]
Therefore, we conclude that $\ell (H) \leq 3 \,c_s r$ and 
${\rm sep} (H, B) \leq 9\sqrt{m}\, c_s r < 3 r /16$.

We now turn to (iv). If $r\geq 2^{-N_0}$, then obviously 
\[
\bE( T_j, \B_r) \leq 
(4\sqrt{m}\, 2^{N_0})^{m+2-2\delta_2} r^{2-2\delta_2} \bE (T_j, \B_{4\sqrt{m}}) \leq  (4\sqrt{m}\, 2^{N_0})^{m+2-2\delta_2} r^{2-2\delta_2} \eps_3^2 \, .
\]
Otherwise, let $k\geq N_0$ be the smallest natural number such that
$2^{-k+1}> r$ and let $L\in \mathscr{sW}^{(j), k}\cup \sS^{(j), k}$ be a cube so that $0\in L$ (cf. \cite[Definition~1.10]{DS4}, $\ell (H) = 2^{-k}$).
By \cite[Proposition~4.2(v)]{DS4}, $|p_L| \leq (\sqrt{m} + C_0 (\bmo^j)^{\sfrac{1}{2m}}) \leq 2 \sqrt{m} \ell (H)$ and so it follows easily that $\bB_r\subset \bB_L$. From condition (Go) we have $L\not\in \sW^{(j)}$.
Thus, by \cite[Proposition~1.11]{DS4}, we get
\[
\bE (T_j, \B_r) \leq C_0 \bE (T_j, \B_L) \leq C_0 \eps_3^2 \,r^{2-2\delta_2}.
\]
Finally, (v) follows from \cite[Corollary~2.2 (ii)]{DS4}, because
by (Go), for every
$r\in ]\frac{s_j}{t_j},3[$, every cube $L\in \sW^{(j)}$ which intersects $B_r (0, \pi_0)$ satisfies $\ell(L) <c_s r$.
\end{proof}

\section{Frequency function and first variations}\label{s:frequency}
Consider the following Lipschitz (piecewise linear) function $\phi:[0+\infty[ \to [0,1]$ given by
\begin{equation*}
\phi (r) :=
\begin{cases}
1 & \text{for }\, r\in [0,\textstyle{\frac{1}{2}}],\\
2-2r & \text{for }\, r\in \,\, ]\textstyle{\frac{1}{2}},1],\\
0 & \text{for }\, r\in \,\, ]1,+\infty[.
\end{cases}
\end{equation*}
For every interval of flattening $I_j = ]s_j, t_j]$,
let $N_j$ be the normal approximation of $T_j$
on $\cM_j$ in \cite[Theorem~2.4]{DS4}.

\begin{definition}[Frequency functions]\label{d:frequency}
For every $r\in ]0,3]$ we define: 
\[
\bD_j (r) := \int_{\cM^j} \phi\left(\frac{d_j(p)}{r}
\right)\,|D N_j|^2(p)\, dp\quad\mbox{and}\quad
\bH_j (r) := - \int_{\cM^j} \phi'\left(\frac{d_j (p)}{r}\right)\,\frac{|N_j|^2(p)}{d(p)}\, dp\, ,
\]
where $d_j (p)$ is the geodesic distance on $\cM_j$ between $p$ and $\Phii_j (0)$.
If $\bH_j (r) > 0$, we define the {\em frequency function}
$\bI_j (r) := \frac{r\,\bD_j (r)}{\bH_j (r)}$. 
\end{definition}

The following is the main analytical estimate of the paper, which allows us to
exclude infinite order of contact among the different sheets of a minimizing
current.

\begin{theorem}[Main frequency estimate]\label{t:frequency}
If $\eps_3$ is sufficiently small, then
there exists a geometric constant $C_0$ such that, for every
$[a,b]\subset [\frac{s_j}{t_j}, 3]$ with $\bH_j \vert_{[a,b]} >0$, we have
\begin{equation}\label{e:frequency}
\bI_j (a) \leq C_0 (1 + \bI_j (b)).
\end{equation}
\end{theorem}

To simplify the notation, in this section we drop the index $j$ and 
omit the measure $\cH^m$ in the integrals over regions of
$\cM$.
The proof exploits four identities collected in Proposition \ref{p:variation},
which will be proved in the next sections.

\begin{definition}\label{d:funzioni_ausiliarie}
We let $\de_{\hat r}$ denote the derivative with respect to arclength along geodesics starting at $\Phii(0)$. We set
\begin{align}
&\qquad\qquad\qquad\bE (r) := - \int_\cM \phi'\left(\textstyle{\frac{d(p)}{r}}\right)\,\sum_{i=1}^Q \langle
N_i(p), \de_{\hat r} N_i (p)\rangle\, dp\,  ,\\
&\bG (r) := - \int_{\cM} \phi'\left(\textstyle{\frac{d(p)}{r}}\right)\,d(p) \left|\de_{\hat r} N (p)\right|^2\, dp
\quad\mbox{and}\quad
\bSigma (r) :=\int_\cM \phi\left(\textstyle{\frac{d(p)}{r}}\right)\, |N|^2(p)\, dp\, .
\end{align}
\end{definition}

\begin{remark}\label{r:tough}Observe that all these functions of $r$ are absolutely continuous
and, therefore, classically differentiable at almost every $r$.
Moreover, the following rough estimate easily follows from
\cite[Theorem~2.4]{DS4} and the condition (Go):
\begin{align}\label{e:rough}
\bD(r) \leq { \int_{\cB_r (\Phii (0))} |DN|^2 \leq }\; C_0\, \bmo\, r^{m+2-2\delta_2} \quad \mbox{for every}\quad r\in\left]\textstyle{\frac{s}{t}},3\right[.
\end{align}
Indeed, since $N$ vanishes identically on the set $\mathcal{K}$ of \cite[Theorem~2.4]{DS4}, it suffices to sum
the estimate of \cite[Theorem~2.4, (2.3)]{DS4} over all the different cubes $L$ (of the corresponding Whitney decomposition)
for which $\Phii (L)$ intersects the geodesic ball $\cB_r$. 
\end{remark}

\begin{proposition}[First variation estimates]\label{p:variation}
For every $\gamma_3$ sufficiently small there is a constant $C = C (\gamma_3)>0$ such that, if
$\eps_3$ is sufficiently small, $[a,b]\subset [\frac{s}{t}, 3]$ and $\bI \geq 1$ on $[a,b]$, then the following
inequalities hold for a.e. $r\in [a,b]$:
\begin{gather}
\left|\bH' (r) - \textstyle{\frac{m-1}{r}}\, \bH (r) - \textstyle{\frac{2}{r}}\,\bE(r)\right|\leq  C \bH (r), \label{e:H'}\\
\left|\bD (r)  - r^{-1} \bE (r)\right| \leq C \bD (r)^{1+\gamma_3} + C \eps_3^2 \,\bSigma (r),\label{e:out}\\
\left| \bD'(r) - \textstyle{\frac{m-2}{r}}\, \bD(r) - \textstyle{\frac{2}{r^2}}\,\bG (r)\right|\leq
C \bD (r) + C \bD (r)^{\gamma_3} \bD' (r) + C r^{-1}\bD(r)^{1+\gamma_3},\label{e:in}\\
\bSigma (r) +r\,\bSigma'(r) \leq C  \, r^2\, \bD (r)\, \leq C r^{2+m} \eps_3^{2}.\label{e:Sigma1}
\end{gather}
\end{proposition}

We assume for the moment the proposition and prove the theorem.

\begin{proof}[Proof of Theorem \ref{t:frequency}.]
Set $\bOmega(r) := \log \big(\max \{\bI(r), 1\}\big)$. Fix a $\gamma_3>0$ and an $\eps_3$ sufficiently small
so that the conclusion
of Proposition \ref{p:variation}. We can thus treat the corresponding constants in the inequalities as geometric ones,
but to simplify the notation we keep denoting them by $C$.

To prove \eqref{e:frequency} it is enough to show $\bOmega (a) \leq C + \bOmega (b)$.
If $\bOmega (a) = 0$, then there is nothing to prove. If $\bOmega (a)>0$, let $b'\in ]a,b]$
be the supremum of $t$ such that $\bOmega>0$ on $]a,t[$.
If $b'<b$, then $\bOmega (b')=0 \leq \bOmega (b)$. Therefore, by possibly
substituting $]a,b[$ with $]a,b'[$, we can assume that $\bOmega >0$, i.e.~$\bI >1$, on $]a,b[$.
By Proposition~\ref{p:variation}, if $\eps_3$ is sufficiently small, then 
\begin{equation}\label{e:out2}
\frac{\bD (r)}{2} 
\stackrel{\eqref{e:out}\,\&\,\eqref{e:Sigma1}}{\leq}
\frac{\bE (r)}{r} 
\stackrel{\eqref{e:out}\,\&\,\eqref{e:Sigma1}}{\leq}
2\, \bD (r),
\end{equation}
from which we conclude that $\bE>0$ over the interval $]a, b'[$.
Set for simplicity $\bF(r) := \bD(r)^{-1} - r \bE(r)^{-1}$, and compute
\[
- \bOmega'  (r) = \frac{\bH' (r)}{\bH(r)} - \frac{\bD'(r)}{\bD(r)} - \frac{1}{r} 
\stackrel{\eqref{e:out}}{=} \frac{\bH'(r)}{\bH(r)} - \frac{r\bD'(r)}{\bE(r)} - \bD'(r) \bF (r) - \frac{1}{r}.
\]
Again by Proposition~\ref{p:variation}:
\begin{equation}\label{e:pezzo_1}
\frac{\bH'(r)}{\bH(r)} \stackrel{\eqref{e:H'}}{\leq} \frac{m-1}{r} + C + \frac2r\,\frac{\bE(r)}{\bH(r)},
\end{equation}
\begin{equation}\label{e:E denominatore}
|\bF(r)| \stackrel{\eqref{e:out}}{\leq} C \, \frac{r (\bD (r)^{1+\gamma_3} + \bSigma (r))}{\bD (r)\,\bE (r)} \stackrel{\eqref{e:out2}}{\leq} C \, \bD (r)^{\gamma_3-1}+ C\,\frac{\bSigma (r)}{\bD (r)^2},
\end{equation}
\begin{align}\label{e:pezzo_2}
- \frac{r\bD'(r)}{\bE(r)} &\stackrel{\eqref{e:in}}{\leq} \left(C - \frac{m-2}{r}\right) \frac{r\bD(r)}{\bE(r)} - \frac2r \, \frac{\bG(r)}{\bE(r)}+  C\frac{r \bD(r)^{\gamma_3} \bD' (r) + \bD (r)^{1+\gamma_3}}{\bE(r)}\nonumber\\
&\leq C - \frac{m-2}{r} 
+ \frac{C}{r} \bD(r) |\bF(r)|
- \frac2r \, \frac{\bG(r)}{\bE(r)}
+  C \bD(r)^{\gamma_3-1} \bD' (r) + C\frac{\bD (r)^{\gamma_3}}{r}\notag\\
& \stackrel{\eqref{e:Sigma1},\, \eqref{e:E denominatore}\,\& \,\eqref{e:rough}}{\leq}C - \frac{m-2}{r} 
- \frac2r \, \frac{\bG(r)}{\bE(r)}
+  C \bD(r)^{\gamma_3-1} \bD' (r) + C\,r^{\gamma_3\,m -1}.
\end{align}
By Cauchy-Schwartz, we have
\begin{equation}\label{e:cauchy-schwartz}
\frac{\bE(r)}{r\bH(r)}\leq \frac{\bG(r)}{r \bE(r)}.
\end{equation}
Thus, by \eqref{e:rough}, \eqref{e:pezzo_1},
\eqref{e:pezzo_2} and  \eqref{e:cauchy-schwartz}, we conclude
\begin{align}
- \bOmega' (r) &\leq C +C\,r^{\gamma_3\,m -1}+
C r \bD(r)^{\gamma_3-1} \bD' (r) 
- \bD'(r) \bF(r)\notag\\
& \stackrel{\eqref{e:E denominatore}}{\leq}C\,r^{\gamma_3\,m -1}+
C\bD(r)^{\gamma_3-1} \bD'(r) + C \frac{\bSigma (r) \bD' (r)}{\bD(r)^2}.\label{e:grosso} 
\end{align}
Integrating \eqref{e:grosso} we conclude:
\begin{align}
\bOmega (a) - \bOmega (b) &\leq C + C \left(\bD(b)^{\gamma_3} - \bD(a)^{\gamma_3}\right) 
+ C \left[ \frac{\bSigma (a)}{\bD (a)} - \frac{\bSigma (b)}{\bD (b)} + \int_a^b \frac{\bSigma' (r)}{\bD (r)}\, dr\right]
\stackrel{\eqref{e:Sigma1}}{\leq} C.\qquad\qedhere\notag
\end{align}
\end{proof}

The rest of the section is devoted to the proof of Proposition~\ref{p:variation}.

\subsection{Estimates on $\bH^\prime$: proof of \eqref{e:H'}} Set $q:= \Phii (0)$.
Let $\exp: B_3\subset T_q\cM \to \cM$ be the exponential map
and $\J \exp $ its Jacobian.
Note that $d (\exp (y),q)=|y|$ for every $y \in B_3$.
By the area formula, setting $y =rz$, we can write
$\bH$ in the following way:
\begin{align*}
\bH (r) & 
= - r^{m-1}\int_{T_q\cM} \frac{\phi'\left(|z|\right)}{|z|}\,|N|^2(\exp (r z))\,\J \exp (r z)\, dx\, .
\end{align*}
Therefore, differentiating under the integral sign, we easily get \eqref{e:H'}:
\begin{align*}
\bH' (r) = {} & - (m-1)\,r^{m-2}\int_{T_q\cM} \frac{\phi'(|z|)}{|z|}\,|N|^2(\exp (r z))\,\J \exp (r z)\, dz
\notag\\
& - 2 \, r^{m-1}\int_{T_q\cM} \phi'(|z|)\,\sum_i \left\langle N_i (\exp (r z)), \de_{\hat{r}} N_i (\exp (r z))\right\rangle\,\J \exp (r z)\, dz
\notag\\
& - r^{m-1}\int_{T_q \cM} \frac{\phi'(|z|)}{|z|}\,|N|^2(\exp (r z))\,\frac{d}{d r} \J \exp (r z)\, dz
\notag\\
= {} & \frac{m-1}{r}\, \bH (r) + \frac{2}{r}\, \bE (r) +O(1)\, \bH (r),
\end{align*}
where we used that $\frac{d}{d r} \J \exp (r\,z) = O(1)$, because $\cM$ is a $C^{3,\kappa}$ submanifold
and hence $\exp$ is a $C^{2, \kappa}$ map (see Proposition \ref{p:exp}). \qed

\subsection{$\bSigma$ and $\bSigma^\prime$: proof of \eqref{e:Sigma1}}
We show the following more precise estimates.

\begin{lemma}\label{l:poincare'} There exists a dimensional constant $C_0>0$ such that
\begin{equation}\label{e:poincare'}
\bSigma (r) \leq C_0\, r^2\,\bD (r) + C_0 r \bH (r) \quad \text{and} \quad 
\bSigma' (r) \leq C_0 \bH (r),
\end{equation}
\begin{equation}\label{e:L2_pieno}
\int_{\mathcal{B}_r (q)} |N|^2 \leq C_0\,\bSigma (r) + C_0\,r\,\bH(r)\, ,
\end{equation}
\begin{equation}\label{e:Dirichlet_pieno}
\int_{\mathcal{B}_r (q)} |DN|^2 \leq C_0\,\bD (r) + C_0\,r \bD' (r).
\end{equation}
In particular, if $\bI \geq 1$, then \eqref{e:Sigma1} holds and
\begin{equation}\label{e:L2_pieno2}
\int_{\mathcal{B}_r (q)} |N|^2 \leq C_0\,r^2\bD (r).
\end{equation}
\end{lemma}

\begin{proof} To simplify the notation we drop the subscript $_0$ from the geometric constants.
Observe that $\psi (p) := \phi \big(\frac{d (p)}{r}\big) |N|^2 (p)$ is a Lipschitz function
with compact support in $\bB_r (q)$. 
We therefore use the Poincar\'e inequality:
$\bSigma (r) = \int_\cM \psi \leq C r \int_\cM |D\psi|$
(the constant $C$ depends on the smoothness of $\cM$).
We compute
\begin{align*}
\bSigma (r) &\leq - C \int_\cM \phi' (r^{-1} d (p)) |N|^2 (p) + C\, r \int_\cM \phi (r^{-1}d (p)) |N| |DN|\\
&\leq Cr \bH (r) + C \bSigma (r)^{\sfrac{1}{2}}\,\big(r^2\bD (r)\big)^{\sfrac{1}{2}} \leq 
C r\bH (r) + \textstyle{\frac{1}{2}} \bSigma (r) + C\, r^2 \bD (r),
\end{align*}
which gives the first part of \eqref{e:poincare'}.
The remaining inequality is straightforward:
\begin{align*}
\bSigma '(r) & = - \int_{\cM} \frac{d(p)}{r^2} \,\phi'\left(\frac{d(p)}{r}\right)\,|N|^2(p)
\leq C \bH (r)\, .
\end{align*}
Since $\phi'=0$ on $]0, \frac{1}{2}[$ and $\phi'= -2$ on $]\frac{1}{2}, 1[$, we easily
deduce
\[
\int_{\cB_r (q)\setminus \cB_{r/2}(q)} |N|^2 \leq r\, \bH(r),
\]
\[
r \bD' (r) = - \int \frac{d (p)}{r} \phi' \left(\frac{d (p, q)}{r}\right) |DN|^2 \geq 
\int_{\cB_r (q)\setminus \cB_{r/2}(q)} |DN|^2\, .
\]
On the other hand, since $\phi=1$ on $[0, \frac{1}{2}]$,
\eqref{e:L2_pieno} and \eqref{e:Dirichlet_pieno} readily follow.
Therefore, in the hypothesis $\bI\geq 1$, i.e.~$\bH \leq r\bD$, we conclude
\eqref{e:Sigma1} from \eqref{e:poincare'}.
\end{proof}

\subsection{First variations}\label{s:variationI}
To prove the remaining estimates in Proposition \ref{p:variation} we 
exploit the first variation of $T$ along
some vector fields $X$. The variations are denoted by $\delta T (X)$.
We fix
a neighborhood $\bU$ of $\cM$ and the normal
projection $\p:\bU\to \cM$ as in \cite[Assumption 2.1]{DS4}. Observe that
$\p\in C^{2,\kappa}$ and \cite[Assumption 3.1]{DS2} holds.
We will consider:
\begin{itemize}
\item the {\em outer variations}, where $X (p)= X_o (p) := \phi \left(\frac{d(\p(p))}{r}\right) \, (p - \p(p) )$.
\item the {\em inner variations}, where $X (p) = X_i (p):= Y(\p(p))$ with
\[
Y(p) := \frac{d(p)}{r}\,\phi\left(\frac{d(p)}{r}\right)\, \frac{\de}{\de \hat r} \quad
\forall \; p \in \cM
\]
($\frac{\partial}{\partial \hat{r}}$ is the unit vector field tangent to the geodesics emanating from $\Phii (0)$
and pointing outwards).
\end{itemize}
Note that $X_i$ is the infinitesimal generator of
a one parameter family of bilipschitz homeomorphisms $\Phi_\eps$ defined as
$\Phi_\eps (p):= \Psi_\eps (\p (p)) + p - \p (p)$, where 
$\Psi_\eps$ is the one-parameter family of bilipschitz homeomorphisms of $\cM$ generated by $Y$.

Consider now the map $F(p) := \sum_i \a{p+ N_i (p)}$ and the current $\bT_F$
associated to its image (cf.~\cite{DS2} for the notation). Observe that $X_i$ and $X_o$ are supported in $\p^{-1} (\cB_r (q))$ but none of them is {\em compactly} supported. 
However, recalling Proposition \ref{p:flattening} (v) and 
the minimizing property of $T$ in $\Sigma$, we deduce that
$\delta T(X) = \delta T (X^T) + \delta T (X^\perp) = \delta T (X^\perp)$,
where $X=X^T+ X^\perp$ is the decomposition of $X$ in the tangent and normal components
to $T\Sigma$.
Then, we have 
\begin{align}\label{e:Err4-5}
&|\delta \bT_F (X)| \leq |\delta \bT_F (X) - \delta T (X)| + |\delta T(X^\perp)|\nonumber\\
\leq& \underbrace{\int_{\supp (T)\setminus \im (F)}  \left|\dv_{\vec T} X\right|\, d\|T\|
+ \int_{\im (F)\setminus \supp (T)} \left|\dv_{\vec \bT_F} X\right|\, d\|\bT_F\|}_{{\rm Err}_4}
+ \underbrace{\left|\int \dv_{\vec T} X^\perp\, d\|T\| \right|}_{{\rm Err}_5}.
\end{align}
Set now for simplicity $\varphi_r (p) :=  \phi \big(\frac{d(p)}{r}\big)$.
We wish to apply \cite[Theorem 4.2]{DS2} to conclude
\begin{align}\label{e:ov graph}
\delta \bT_F (X_o) =& \int_\cM\Big( \ph_r\,|D N|^2 + \sum_{i=1}^Q N_i\otimes \nabla \ph_r : D N_i \Big) + \sum_{j=1}^3 \textup{Err}^o_j,
\end{align}
where the errors ${\rm Err}^o_j$ correspond to the terms ${\rm Err_j}$ of \cite[Theorem 4.2]{DS2}. This
would imply
\begin{gather}
{\rm Err}_1^o = - Q \int_\cM \varphi_r \langle H_\cM, \etaa\circ N\rangle,\label{e:outer_resto_1}\\
|{\rm Err}_2^o| \leq C_0 \int_\cM |\varphi_r| |A|^2|N|^2,\label{e:outer_resto_2}\\
|{\rm Err}_3^o| \leq C_0 \int_\cM \big(|N| |A| + |DN|^2 \big) \big( |\varphi_r| |DN|^2  + |D\varphi_r| |DN| |N| \big)\label{e:outer_resto_3}\,,
\end{gather}
where $H_\cM$ is the mean curvature vector of $\cM$.
Note that \cite[Theorem 4.2]{DS2} requires the $C^1$ regularity of $\varphi_r$.
We overcome this technical obstruction applying
\cite[Theorem 4.2]{DS2} to a standard smoothing of $\phi$ and then passing into the limit (the obvious details are left to the reader). Plugging \eqref{e:ov graph} into \eqref{e:Err4-5}, we then conclude
\begin{equation}\label{e:ov_con_errori}
\left| \bD (r) - r^{-1} \bE (r)\right| \leq \sum_{j=1}^5 \left|\textup{Err}^o_j\right|\, ,
\end{equation}
where ${\rm Err}_4^o$ and ${\rm Err}_5^o$ correspond respectively 
to ${\rm Err}_4$ and
${\rm Err}_5$ of \eqref{e:Err4-5} when $X=X_o$.
With the same argument, but applying this time \cite[Theorem 4.3]{DS2} to $X=X_i$,
we get
\begin{align}\label{e:iv graph}
\delta \bT_F (X_i) =& \frac{1}{2} \int_\cM\Big(|D N|^2 {\rm div}_{\cM} Y  - 2
\sum_{i=1}^Q  \langle D N_i : ( D N_i \cdot D_{\cM} Y)\rangle \Big) + \sum_{j=1}^3 \textup{Err}^i_j\, ,
\end{align}
where this time the errors ${\rm Err}^i_j$ correspond to the error terms ${\rm Err}_j$ of 
\cite[Theorem 4.3]{DS2}, i.e.
\begin{gather}
{\rm Err}_1^i = - Q \int_{ \cM}\big( \langle H_\cM, \etaa \circ N\rangle\, {\rm div}_{\cM} Y + \langle D_Y H_\cM, \etaa\circ N\rangle\big)\, ,\label{e:inner_resto_1}\\
|{\rm Err}_2^i| \leq C_0 \int_\cM |A|^2 \left(|DY| |N|^2  +|Y| |N|\, |DN|\right), \label{e:inner_resto_2}\\
|{\rm Err}_3^i|\leq C_0 \int_\cM \Big( |Y| |A| |DN|^2 \big(|N| + |DN|\big) + |DY| \big(|A|\,|N|^2 |DN| + |DN|^4\big)\Big)\label{e:inner_resto_3}\, .
\end{gather}
Straightforward computations (again appealing to Proposition \ref{p:exp}) lead to 
\begin{align}\label{e:DY}
D_{\cM} Y (p) = \phi'\left(\frac{d(p)}{r}\right) \, \frac{d(p)}{r^2} \, \frac{\de}{\de \hat r} \otimes \frac{\de}{\de \hat r} + \phi \left(\frac{d(p)}{r}\right) \left( \frac{\Id}{r} + O(1)\right)\, ,
\end{align}
\begin{align}\label{e:divY}
\dv_\cM\, Y (p) 
& = \phi'\left(\frac{d(p)}{r}\right) \, \frac{d(p)}{r^2} + \phi\left(\frac{d(p)}{r}\right) \, \left(\frac{m}{r} + O(1) \right)\, .
\end{align}
Plugging \eqref{e:DY} and \eqref{e:divY} into \eqref{e:iv graph} and using \eqref{e:Err4-5} we then
conclude
\begin{equation}\label{e:iv_con_errori}
\left| \bD' (r) - (m-2)r^{-1} \bD (r) - 2 r^{-2} \bG (r)\right|
\leq C_0 \bD (r) + \sum_{j=1}^5 \left|{\rm Err}^i_j\right|\, .
\end{equation}
Proposition~\ref{p:variation} is then proved by the estimates of the errors
terms done in the next section.
 
\section{Error estimates}\label{s:variationII}
We start with some preliminary considerations, keeping the notation and convention of the previous section (and
dropping the subscript when dealing with the maps of Theorem  \ref{t:frequency} and Proposition \ref{p:variation}).

\subsection{Families of subregions}\label{s:regions}
Set $q:= \Phii (0)$.
We select a family of subregions of $\cB_r (p)\subset \cM$.
Denote by $B$ and $\partial B$ respectively $\p_\pi (\cB_r (q))$
and $\p_\pi (\partial \cB_r (q))$, where $\pi$ is the reference $m$-dimensional
plane of the construction of the center manifold $\cM$.
Since $\|\phii\|_{C^{3,\kappa}}\leq C \eps_3^{1/m}$ (cf.~\cite[Theorem~1.17]{DS4}),
by Proposition \ref{p:exp}
we can assume that $B$ is a $C^2$ convex set which at any boundary point $p$ contains an interior
sphere of radius $r/2$ passing through $p$. Thus:
\begin{equation}\label{e:convex}
\forall z\in \de B \quad \mbox{there is a ball $B_{r/2} (y, \pi)\subset B$ whose closure touches
$\partial B$ at $z$.}
\end{equation}

\begin{definition}[Family of cubes]\label{d:cubes}
We first define a family $\mathcal{T}$ of cubes in the Whitney decomposition 
$\mathscr{W}$ as follows:
\begin{itemize}
\item[(i)] $\mathcal{T}$ includes all $L\in \mathscr{W}_h \cup \mathscr{W}_e$ which intersect $B$;
\item[(ii)] if $L'\in \mathscr{W}_n$ intersects $B$ and
belongs to the domain of influence
$\mathscr{W}_n (L)$ of the cube $L\in \mathscr{W}_e$ as in \cite[Corollary~3.2]{DS4}, then $L\in \mathcal{T}$.
\end{itemize}
\end{definition}

\begin{definition}[Associated balls $B^L$]\label{d:balls}
By Proposition~\ref{p:flattening} (iii), $\ell (L)\leq 3 c_s r \leq r$
and ${\rm sep} (L, B) \leq 3\sqrt{m} \,\ell (L)$ for each $L\in \mathcal{T}$. 
Let $x_L$ be the center of $L$ and:
\begin{itemize}
\item[(a)] if $x_L \in \overline{B}$, we then set $s(L):=\ell(L)$ and
$B^L := B_{s(L)} (x_L, \pi)$;
\item[(b)] otherwise we consider the ball $B_{r(L)} (x_L, \pi)\subset \pi$ whose closure
touches $\overline{B}$ at exactly one point $p(L)$, we set $s(L):= r(L) + \ell (L)$
and define $B^L:= B_{s(L)} (x_L, \pi)$.
\end{itemize}
\end{definition}
{ Observe that, when $L\in \mathcal{T}\cap \sW_h$, then $s (L)$ is at most $(\sqrt{m}+ 1) \ell (L)$.}
We proceed to select a countable family
$\mathscr{T}$ of pairwise disjoint balls $\{B^L\}$. We
let $S:= \sup_{L\in \mathcal{T}} s(L)$ and start selecting
a maximal subcollection $\mathscr{T}_1$ 
of pairwise disjoint balls with radii larger than $S/2$.
Clearly, $\mathscr{T}_1$ is finite. In general, at the stage $k$,
we select a maximal subcollection $\mathscr{T}_k$ of pairwise disjoint balls
which do not intersect any of the previously selected balls in
$\mathscr{T}_1 \cup \ldots \cup \mathscr{T}_{k-1}$ and which
have radii $r\in ]2^{-k} S, 2^{1-k}S]$.
Finally, we set $\mathscr{T} := \bigcup_k \mathscr{T}_k$. 

\begin{definition}[Family of cube-ball pairs $(L, B(L))\in \mathscr{Z}$]\label{d:coppie}
Recalling \eqref{e:convex} and $\ell (L)\leq r$, it easy to see that there exist balls
$B_{\ell (L)/4} (q_L, \pi) \subset B^L\cap B$ which lie at distance
at least $\ell (L)/4$ from $\partial B$. We denote by $B(L)$ one such ball and 
by $\mathscr{Z}$ the collection of pairs $(L, B(L))$ with $B^L\in \mathscr{T}$.
\end{definition}

Next, we partition the cubes of $\mathscr{W}$ which intersect $B$
into disjoint families $\mathscr{W} (L)$ labeled by $(L,B(L))\in \mathscr{Z}$
in the following way (observe that $\mathscr{W} (L)$ and $\mathscr{W}_n (L)$ are different families
and should not be confused!).
Let $H\in \mathscr{W}$ have nonempty intersection with $B$.
If $H$ is itself in $\mathcal{T}$, we then select $L\in \mathscr{T}$ with $B^L\cap B^H \neq\emptyset$ and assign 
$H\in \mathscr{W} (L)$. 
Otherwise $H$ is in the domain of influence of
some $J\in \mathscr{W}_e$.
By Proposition~\ref{p:flattening}, the separation between $J$ and $H$ is at most
$3\sqrt{m} \ell (J)$ and, hence, $H\subset B_{4\sqrt{m} \ell (J)} (x_J)$.
By construction there is a $B^L \in \mathscr{T}$ with $B^J\cap B^L\neq \emptyset$
and radius $s(L)\geq \frac{s(J)}{2}$. We then prescribe $H\in \mathscr{W} (L)$.
Observe that $s(L)\leq 4\sqrt{m} \ell (L)$ and $s(J) \geq \ell (J)$.
Therefore, $\ell (J) \leq 8\sqrt{m} \ell (L)$ and $|x_J - x_L|\leq 5 s(L) \leq 20 \sqrt{m} \ell (L)$. This implies that
\[
H\subset B_{4\sqrt{m} \ell (J)} (x_J) \subset B_{4\sqrt{m} \ell (J) + 20 \sqrt{m} \ell (L)} (x_L)
\subset B_{30 \sqrt{m} \ell (L)} (x_L)\, .
\]
The inclusion $H\subset B_{30 \sqrt{m} \ell (L)} (x_L)$ holds also in case $H\in \mathcal{T}$, as can be easily seen
simply setting $J=H$ and using the same computations.
For later reference, we collect the main properties of the above construction.

\begin{lemma}\label{l:covering}
The following holds.
\begin{enumerate}
\item[(i)] If $(L, B(L))\in \mathscr{Z}$, then $L\in \mathscr{W}_e \cup \mathscr{W}_h$, the radius
of $B(L)$ is $\frac{\ell (L)}{4}$, $B(L)\subset B^L\cap B$
and ${\rm sep}\, (B(L), \partial B)\geq \frac{\ell (L)}{4}$.

\item[(ii)] If the pairs $(L, B(L)), (L', B(L'))\in \mathscr{Z}$ are distinct, then
$L$ and $L'$ are distinct and $B (L)\cap B (L') =\emptyset$.

\item[(iii)] The cubes $\mathscr{W}$ which intersect $B$
are partitioned into disjoint families $\mathscr{W} (L)$ labeled by
$(L,B(L))\in \mathscr{Z}$ such that,
if $H\in \mathscr{W} (L)$, then $H\subset B_{30 \sqrt{m} \ell (L)} (x_L)$. 
\end{enumerate}
\end{lemma}

\subsection{Basic estimates in the subregions}\label{s:est_regions}
For notational convenience, we order the family 
$\mathscr{Z} =\{(L_i, B(L_i))\}_{i\in \N}$, and set
\[
\mathcal{B}^i := \Phii (B (L_i))\qquad \mathcal{U}_i = \cup_{H\in \mathscr{W} (L_i)} \Phii (H) \cap \cB_r (q)\, 
\] 
(recall that $q= \Phii (0)$).
Observe that the separation between $\cB^i$ and
$\partial \mathcal{B}_r (q)$ is larger than that between $B(L_i)$ and $\partial B = \p_\pi (\partial \mathcal{B}_r (q))$. Thus, by Lemma~\ref{l:covering} (i),
$\varphi_r (p) = \phi \big(\frac{d(p)}{r}\big)$ satisfies
\begin{align}
\inf_{p\in \cB^i} \varphi_r (p) \geq (4r)^{-1} \ell_i\label{e:peso_1}\, ,
\end{align}
where $\ell_i := \ell (L_i)$. From this and Lemma~\ref{l:covering} (iii), 
we also obtain
\[
\sup_{p\in \mathcal{U}_i} \varphi_r (p)
- \inf_{p\in \mathcal{U}_i}\varphi_r (p) \leq C \Lip (\varphi_r) \ell_i \leq \frac{C}{r}\,\ell_i \stackrel{\eqref{e:peso_1}}{\leq}
C \inf_{p\in \mathcal{B}_i} \varphi_r (p)\, ,
\]
which translates into
\begin{equation}\label{e:peso_2}
\sup_{p\in \mathcal{U}_i} \varphi_r (p) \leq C \inf_{p\in \cB^i} \varphi_r (p)\, .
\end{equation}
Moreover, set { $\mathcal{V}_i := \mathcal{U}_i \cap (((\supp (\bT_F) \setminus \supp (T))\cup (\supp (T) \setminus \supp (\bT_F)))$ and observe that $\mathcal{V}_i \subset \mathcal{U}_i\setminus \mathcal{K}$,} where $\mathcal{K}$ is the coincidence set of \cite[Theorem~2.4]{DS4}.
From \cite[Theorem~2.4]{DS4}, we derive the following estimates:
\begin{gather}
\int_{\mathcal{U}_i} |\etaa\circ N| \leq C_0 \bmo\,\ell_i^{2+m+\sfrac{\gamma_2}{2}} + C_0 \int_{\cU^i} |N|^{2+\gamma_2},\label{e:media}\\
\int_{\mathcal{U}_i} |DN|^2 \leq C_0 \bmo\, \ell_i^{m+2-2\delta_2},\label{e:Dirichlet_sopra}\\
\|N\|_{C^0 (\mathcal{U}_i)} + \sup_{p\in \supp (T) \cap \p^{-1} (\mathcal{U}_i)} |p- \p (p)| \leq C_0 \bmo^{\sfrac{1}{2m}} \ell_i^{1+\beta_2},\label{e:N_sopra}\\
\Lip (N|_{\mathcal{U}_i}) \leq C_0 \bmo^{\gamma_2} \ell_i^{\gamma_2},\label{e:Lipschitz}\\
\mass (T\res \p^{-1} (\mathcal{V}_i)) + \mass (\bT_F \res \p^{-1} (\mathcal{V}_i)) \leq C_0 \bmo^{1+\gamma_2} \ell_i^{m+2+\gamma_2}.\label{e:errori_massa}
\end{gather}
To prove these estimates, observe first that
$\sum_{H\in \sW (L_i)} \ell (H)^m \leq C_0 \ell_i^m$,
because all $H\in \sW (L_i)$ are disjoint and contained in a ball of radius
comparable to $\ell_i$. This in turn implies that
$\sum_{H\in \sW (L_i)} \ell (H)^{m+\varepsilon} \leq C_0 \ell_i^{m+\varepsilon}$, because 
$\ell (H)\leq \ell_i$ for any $H\in \sW (L)$. Thus:
\begin{itemize}
\item[-] \eqref{e:media} follows summing the estimate of
\cite[Theorem~2.4 (2.4)]{DS4} applied with $a=1$ to 
$\Phii (H)$ with $H\in \sW (L_i)$;
\item[-] \eqref{e:Dirichlet_sopra} follows from summing the estimate of \cite[Theorem~2.4 (2.3)]{DS4} applied to 
$\Phii (H)$ with $H\in \sW (L_i)$;
\item[-] \eqref{e:N_sopra} follows from \cite[Theorem~2.4 (2.1)]{DS4} and
\cite[Corollary~2.2 (ii)]{DS4};
\item[-] \eqref{e:Lipschitz} follows from \cite[Theorem~2.4 (2.1)]{DS4};
\item[-] \eqref{e:errori_massa} follows summing \cite[Theorem~2.4 (2.2)]{DS4}
applied to $\Phii (H)$ with $H\in \sW (L_i)$.
\end{itemize}

The last ingredient for the completion of the proof of Proposition~\ref{p:variation}
are the following three key estimates which are derived from the analysis of
the construction of the center manifold in \cite{DS4}.

\begin{lemma}\label{l:exponent_match}
Under the assumptions of Proposition \ref{p:variation}, it holds
\begin{align}
\sum_i \big(\inf_{\cB^i} \varphi_r \big) \,\bmo \,
\ell_i^{m+2+\sfrac{\gamma_2}{4}} \leq C_0 \bD (r),\label{e:D_globale}\\
\sum_i \bmo\, \ell_i^{m+2+\sfrac{\gamma_2}{4}} \leq C_0 \left( \bD (r) + r \bD' (r)\right)\, ,\label{e:D'_globale}
\end{align}
for some geometric constant $C_0$.
Moreover, for every $t>0$ there exists $C_0>0$ and $a>0$ such that, for $C(t) = C^t$ and $\gamma (t) = a t$ we have
\begin{equation}\label{e:gamma(a)}
\sup_i
\bmo^t\Big[{\ell_i}^t + \Big(\inf_{\cB^i} \varphi_r\Big)^{\sfrac{t}{2}} \ell_i^{\sfrac{t}{2}}\Big]
\leq C(t) \bD (r)^{\gamma(t)}.
\end{equation}
\end{lemma}

\begin{proof} Recall that, from \cite[Propositions~3.1 and 3.4]{DS4} and \eqref{e:peso_1} 
we have, for some geometric positive constant $c_0$
{
\begin{align}
\int_{\cB^i} \varphi_r |N|^2 &\geq c_0 \, \bmo^{\sfrac{1}{m}} \inf_{\cB^i} \varphi_r \ell_i^{m+2+2\beta_2}\geq c_0 \bmo^{\sfrac{1}{a}} \Big[\ell_i^2 + \Big(\inf_{\cB^i} \varphi_r\Big) \ell_i\Big]^{\sfrac{1}{(2a)}}  \quad 
\mbox{if $L_i\in \mathscr{W}_h$},\label{e:N_sotto}\\
\int_{\cB^i} \varphi_r |DN|^2 &\geq c_0 \, \bmo\, \inf_{\cB^i} \varphi_r \ell_i^{m+2-2\delta_2}\geq c_0 \bmo^{\sfrac{1}{a}} \Big[\ell_i^2 + \Big(\inf_{\cB^i} \varphi_r\Big) \ell_i\Big]^{\sfrac{1}{(2a)}}
\quad\mbox{if $L_i \in \mathscr{W}_e$}\, \label{e:Dirichlet_sotto}
\end{align}
where we just need $a\leq \min \{ 1/(2 (m+2+2\beta_2)), 1/(2(m+2-2\delta_2))\}$
(note that \eqref{e:N_sotto} follows from \cite[Proposition~3.1 (S3)]{DS4} because $s (L_i)\leq (\sqrt{m}+1) \ell (L_i)$  for $L_i\in \mathscr{W}_h$).
Therefore, by Lemma~\ref{l:poincare'}, \eqref{e:peso_1}, \eqref{e:N_sotto} and
\eqref{e:Dirichlet_sotto}, it follows easily that,
\begin{align*}
&2^{-t} c_0^{at} \bmo^t\Big[\ell_i^t + \Big(\inf_{\cB^i} \varphi_r\Big)^{\sfrac{t}{2}} \ell_i^{\sfrac{t}{2}}\Big]
\leq \left(\int_{\cB^i} \varphi_r |DN|^2\right)^{at} +\left(\int_{\cB^i} \varphi_r |N|^2\right)^{at}\\
 \leq & 2^t \left(\int_{\cB^i} \varphi_r (|DN|^2 + |N|^2)\right)^{at}
\stackrel{\eqref{e:L2_pieno2}\,\&\,\bI\geq1}{\leq} C_0^t \bD (r)^{at}\, .
\end{align*}
Taking the supremum over $i$ we achieve \eqref{e:gamma(a)}}. Next, \eqref{e:D_globale} follows similarly because the $\cB^i$ are disjoint and $8\,\beta_2<\gamma_2$:
\[
\sum_i \big(\inf_{\cB^i} \varphi_r \big) \bmo\, \ell_i^{m+2+\sfrac{\gamma_2}{4}} 
\leq C \sum_i \int_{\cB^i} \varphi_r (|DN|^2 +|N|^2) 
\stackrel{\eqref{e:L2_pieno2}\,\&\,\bI\geq1}{\leq} C \bD (r)\, .
\]
Finally, arguing as above we conclude that
\begin{equation}\label{e:nuova_referenza}
\sum_i \bmo\, \ell_i^{m+2+\sfrac{\gamma_2}{4}} \leq C \int_{\cB_r (q)} \big(|DN|^2+|N|^2\big)
\stackrel{\eqref{e:Dirichlet_pieno}\,\&\,\eqref{e:L2_pieno2}}{\leq} C \big(\bD (r) + r \bD' (r)\big).
\end{equation}
and, hence, \eqref{e:D'_globale} follows from Lemma \ref{l:poincare'}.
\end{proof}

\subsection{Proof of Proposition~\ref{p:variation}: \eqref{e:out} and \eqref{e:in}}
We can now pass to
estimate the errors terms in \eqref{e:out} and \eqref{e:in} in order to conclude
the proof of Proposition~\ref{p:variation}. Unless otherwise specified, the constants denoted by $C$ will be
assumed to be geometric (i.e. to depend only upon the parameters introduced in \cite{DS4}). 

\medskip

\noindent {\bf Errors of type 1.}
By \cite[Theorem~1.12]{DS4}, the map $\phii$ defining the center
manifold satisfies $\|D\phii\|_{C^{2, \kappa}} \leq C\, \bmo^{\sfrac12}$,
which in turn implies $\|H_\cM\|_{L^\infty} + \|DH_\cM\|_{L^\infty} \leq C\, \bmo^{\sfrac12}$ (recall that $H_\cM$ denotes
the mean curvature of $\cM$).
Therefore, by \eqref{e:peso_2}, \eqref{e:media}, \eqref{e:D_globale} and \eqref{e:gamma(a)}, we get
\begin{align}
\left|{\rm Err}^o_1\right| &\leq C_0 \int_\cM \varphi_r \,|H_\cM|\,
|\etaa\circ N| \notag\allowdisplaybreaks\\
& \leq C_0 \, \bmo^{\sfrac12} \sum_j  
\Big(\big(\sup_{\mathcal{U}_i} \varphi_r\big)\,\bmo\, \ell_j^{2+m+\gamma_2} + C_0 \int_{\mathcal{U}_j} \ph_r|N|^{2+\gamma_2}\Big)\notag\allowdisplaybreaks\\
&\leq C \bD (r)^{1+\gamma_3} 
+ C \sum_j \bmo^{\sfrac12} \,\ell_j^{\gamma_2(1+\beta_2)} \int_{\mathcal{U}_j} \ph_r|N|^2
\leq C (\gamma_3) \bD (r)^{1+\gamma_3}\,,\notag\allowdisplaybreaks
\end{align}
provided $\gamma_3> 0$ is sufficiently small depending only upon $m, \beta_2, \delta_2$ and $\gamma_2$. Analogously
\begin{align}
& \left|{\rm Err}^i_1\right| \leq C\,r^{-1} \int_\cM \big(|H_\cM|+ |D_Y H_\cM|\big)
|\etaa\circ N| \notag\allowdisplaybreaks\\
\leq & C \,r^{-1} \bmo^{\sfrac12} \sum_j  
\Big(\,\bmo\, \ell_j^{2+m+\gamma_2} + C \int_{\mathcal{U}_j} |N|^{2+\gamma_2}\Big)
\leq C (\gamma_3)\,r^{-1} \bD (r)^{\gamma_3}\big(\bD(r)+r\,\bD'(r)\big).\notag
\end{align}

\noindent {\bf Errors of type 2.}
From $\|A\|_{C^0} \leq C \|D \phii\|_{C^2} \leq C \bmo^{\sfrac{1}{2}} \leq C \eps_3$, it follows that
${\rm Err}^o_2 \leq C \eps_3^2 \bSigma (r)$.
Moreover, since $|D X_i| \leq Cr^{-1}$, Lemma~\ref{l:poincare'} gives
\begin{align}
\left|{\rm Err}^i_2\right| &\leq C r^{-1} \int_{\cB_r (p_0)} |N|^2 + C \int \varphi_r |N| |DN|
\leq C \bD (r)\, .\notag
\end{align}

\noindent {\bf Errors of type 3.}
Clearly, we have
\begin{align}
\left|{\rm Err}^o_3\right| \leq & \underbrace{\int \varphi_r \left(|DN|^2 |N| + |DN|^4\right)}_{I_1}
+ C\underbrace{r^{-1} \int_{\cB_r (q)} |DN|^3 |N|}_{I_2}
 + C\underbrace{r^{-1} \int_{\cB_r (q)} |DN| |N|^2}_{I_3}\, .\notag
\end{align}
We estimate separately the three terms (recall that
$\gamma_2> 4\delta_2$):
\begin{align}
I_1 & \leq \int_{\cB_r (p_0)} \varphi_r (|N|^2 |DN| + |DN|^3)
\leq I_3 +  C \sum_j \sup_{\cU_j} \varphi_r \bmo^{1+\gamma_2} \ell_j^{m+2+\sfrac{\gamma_2}{2}}\nonumber\allowdisplaybreaks\\
&\stackrel{\eqref{e:D_globale}\,\&\,\eqref{e:gamma(a)}}{\leq} I_3 + C (\gamma_3) \bD (r)^{1+\gamma_3},\notag
\allowdisplaybreaks\\
I_2 &\leq C r^{-1} \sum_j \bmo^{1+\sfrac{1}{2m}+\gamma_2} \ell_j^{m+3+\beta_2 +\sfrac{\gamma_2}{2}}
\stackrel{\eqref{e:peso_2}}{\leq} C \sum_j \bmo^{1+\sfrac{1}{2m}+\gamma_2} \ell_j^{m+2+\beta_2 +\sfrac{\gamma_2}{2}} \inf_{\cB^j} \varphi_r\nonumber\allowdisplaybreaks\\
&\stackrel{\eqref{e:D_globale}\,\&\,\eqref{e:gamma(a)}}{\leq} C (\gamma_3) \bD (r)^{1+\gamma_3},\notag
\allowdisplaybreaks\\
I_3 &\leq C r^{-1} \sum_j \bmo^{\gamma_2} \ell_j^{\gamma_2} \int_{\mathcal{U}_j} |N|^2
\stackrel{\eqref{e:gamma(a)}}{\leq} C \gamma_3) r^{-1} \bD (r)^{\gamma_3}  \int_{\cB_r (q)} |N|^2
\stackrel{\eqref{e:L2_pieno2}}{\leq} C (\gamma_3) \bD (r)^{1+\gamma_3}\, ,\notag
\end{align}
provided $\gamma_3>0$ is sufficiently small.
For what concerns the inner variations, we have 
\begin{align}
|{\rm Err}^i_3| &\leq C \int_{\cB_r (q)} \big(r^{-1}|DN|^3 + r^{-1} |DN|^2 |N| + r^{-1} |DN| |N|^2\big)\, .\notag
\end{align}
The last integrand corresponds to $I_3$, while the remaining part can be estimated as
follows:
\begin{align}
\int_{\cB_r (q)} r^{-1}(|DN|^3 + |DN|^2 |N|)&\leq C \sum_j r^{-1}(\bmo^{\gamma_2} \ell_j^{\gamma_2} +
\bmo^{\sfrac{1}{2m}} \ell_j^{1+\beta_2}) \int_{\mathcal{U}_j} |DN|^2\notag\allowdisplaybreaks\\
&\stackrel{\mathclap{\eqref{e:gamma(a)}}}{\leq} \, C (\gamma_3)\,r^{-1} \bD (r)^{\gamma_3} \int_{\cB_r (q)} |DN|^2\notag\allowdisplaybreaks\\
&\leq 
C(\gamma_3)  \bD (r)^{\gamma_3} \left(\bD' (r) + r^{-1}\bD (r)\right)\, .\notag
\end{align}

\medskip

\noindent {\bf Errors of type 4.} We compute explicitly
\begin{align}
|D X_o (p)| &\leq 2\, |p - \p(p)| \, \frac{|Dd(\p(p),q)|}{r} + \varphi_r (p) \, |D(p-\p(p))|
\leq C\,\left(\frac{|p - \p(p)|}{r} + \varphi_r (p)\right)\, .\notag
\end{align}
It follows readily from \eqref{e:Err4-5}, \eqref{e:N_sopra}
and \eqref{e:errori_massa} that
\begin{align}
|{\rm Err}^o_4| &\leq \sum_i C \Big(r^{-1} \bmo^{\sfrac{1}{2m}} \ell_i^{1+\beta_2}
+ \sup_{\mathcal{U}_i} \varphi_r\Big) \bmo^{1+\gamma_2} \ell_i^{m+2+\gamma_2}\notag\allowdisplaybreaks\\
&\stackrel{\eqref{e:peso_1}\,\& \,\eqref{e:peso_2}}{\leq} C \sum_i
\left[
\bmo^{\gamma_2} \ell_i^{\sfrac{\gamma_2}{4}}\right] 
\inf_{\cB_i} \varphi_r\, \bmo\, \ell_i^{m+2+\sfrac{\gamma_2}{4}}
\stackrel{\eqref{e:D_globale}\,\&\eqref{e:gamma(a)}}{\leq} C (\gamma_3) \bD (r)^{1+\gamma_3}.
\label{e:stima4-5o}
\end{align}
Similarly, since $|D X_i|\leq C r^{-1}$, we get
\begin{align*}
{\rm Err}^i_4 &\leq C r^{-1} \sum_j \Big(\bmo^{\gamma_2} \ell_j^{\sfrac{\gamma_2}{2}}\Big)
\bmo\,\ell_j^{m+2+\sfrac{\gamma_2}{2}}
\stackrel{\eqref{e:D'_globale}\,\&\,\eqref{e:gamma(a)}}{\leq} C (\gamma_3) \bD (r)^{\gamma_3} \left(\bD' (r) +r^{-1} \bD(r)\right)\, .\notag
\end{align*}

\noindent {\bf Errors of type 5.}
Integrating by part Err$_5$, we get
\begin{align*}
{\rm Err}_5 =& \left| \int \langle X^\perp, h(\vec{T}(p)) \rangle d\|T\|  \right| \leq \underbrace{\left| \int \langle X^\perp, h(\vec{\bT}_F(p)) \rangle d\|\bT_F\| \right|}_{I_2}\\
&+ \underbrace{\int_{\supp(T)\setminus \im(F)} |X^\perp| |h(\vec{T}(p))| d\|T\|
+\int_{\im (F)\setminus \supp (T)} |X^\perp| |h (\vec{\bT}_F (p)| d\|\bT_F\|}_{I_1},
\end{align*}
where $h(\vec \lambda)$ is the trace of $A_\Sigma$ on the $m$-vector $\vec \lambda$,
i.e.~$h(\vec \lambda) := \sum_{k=1}^m A_\Sigma(v_k, v_k)$ with $v_1, \ldots, v_m$ 
orthonormal vectors such that
$v_1\wedge\ldots\wedge v_m = \vec \lambda$.

Since $|X|\leq C$, $I_1$ can be easily estimated as ${\rm Err}_4$:
\begin{align*}
I_2 \leq C \sum_j (\sup_{\cU_i} \ph_r )
\bmo^{1+\gamma_2}\,\ell_j^{m+2+\gamma_2}\leq C (\gamma_3) \bD^{1+\gamma_3}(r).
\end{align*}
For what concerns $I_2$, we argue differently for the outer and the inner variations.
For ${\rm Err}_5^o$, observe that $|X^{o\perp} (p)| = \varphi_r (\p (p)) |\p_{T_p \Sigma^\perp} (p-\p (p))|$.
On the other hand, we also have
\[
|\p_{T_p \Sigma^\perp} (p-\p (p))| \leq C \mathbf{c} (\Sigma) |p-\p (p)|^2 \leq C \bmo^{\sfrac{1}{2}} |p-\p (p)|^2\quad\forall\;p\in\Sigma.
\] 
Therefore, we can estimate
\begin{align*}
I_2^o & \leq C\, \bmo \int \ph_r |N|^2 \leq C\, \eps_3^2 \bSigma(r).
\end{align*}
For the inner variations, denote by $\nu_1, \ldots, \nu_l$
an orthonormal frame for $T_p\Sigma^\perp$ of class $C^{2,\eps_0}$ (cf.~\cite[Appendix A]{DS2}) and set 
$h^j_p (\vec \lambda) := - \sum_{k=1}^m \langle D_{v_k}\nu_j (p),  v_k \rangle$
whenever $v_1\wedge\ldots\wedge v_m = \vec \lambda$ is an $m$-vector of $T_p \Sigma$ (with 
$v_1, \ldots, v_m$ orthonormal). 
For the sake of simplicity, we write
\[
h^j_p := h^j_p (\vec\bT_F (p)) \quad \text{and}\quad
h_p = \sum_{j=1}^l h^j_p  \nu_j(p),
\]
\[
h^j_{\p(p)} := h^j_{\p(p)} (\vec \cM (\p (p)))
\quad \text{and}\quad
h_{\p(p)} = \sum_{j=1}^l h^j_{\p(p)} \nu_j(\p(p)).
\]
Consider the exponential map $\mathbf{ex}_{\mathbf{p} (p)}: T_{\mathbf{p} (p)} \Sigma \to \Sigma$ and its inverse $\mathbf{ex}^{-1}_{\mathbf{p} (p)}$. Recall that:
\begin{itemize}
\item the geodesic distance $d_\Sigma (p, q)$ is comparable to $|p-q|$ up to a constant factor;
\item $\nu_j$ is $C^{2, \eps_0}$ and $\|D\nu_j\|_{C^{1, \eps_0}} \leq C \bmo^{\sfrac{1}{2}}$;
\item $\mathbf{ex}_{\p (p)}$ and $\mathbf{ex}^{-1}_{\p (p)}$ are both $C^{2, \eps_0}$ and $\|{\rm d}\,\mathbf{ex}_{\p (p)}\|_{C^{1, \eps_0}} + \|{\rm d}\,\mathbf{ex}^{-1}_{\p (p)}\|_{C^{1, \eps_0}} \leq \bmo^{\sfrac{1}{2}}$;
\item $|h^j_p| \leq  C\|A_\Sigma\|_{C^0} \leq C \bmo^{\sfrac12}$;
\end{itemize} 
where all the constants involved are just geometric.
We then conclude that
\begin{align}
& h_p - h_{\p(p)} = \sum_{j} \nu_j(p) (h_p^j - h^j_{\p(p)})+
\sum_{j} \big(\nu_j(p) - \nu_j(\p(p))\big) h_{\p(p)}^j\notag\\
= & \sum_{j} \nu_j(p) (h_p^j - h^j_{\p(p)})
+ \sum_{j} D\nu_j(\p(p))\cdot \mathbf{ex}^{-1}_{\p (p)}(p)\, h_{\p(p)}^j+ O(|p-\p(p)|^2).
\end{align}
On the other hand, $X_i(p) = Y(\p(p))$ is tangent to $\cM$ in $\p (p)$ and hence orthogonal to
$h_{\p (p)}$. Thus
\begin{align}\label{e:pezzo lineare}
& \langle X_i(p), h_p \rangle  = {} \langle X^i(p), (h_p - h_{\p(p)}) \rangle = \sum_j \langle X_i(\p(p)), D\nu_j(\p(p))\cdot \mathbf{ex}^{-1}_{\p (p)} (p)\rangle h^j_{\p(p)}\notag\\
& \qquad\qquad\qquad\qquad\qquad\qquad\qquad\qquad + \sum_j  \langle \nu_j (p), X_i (p)\rangle \big(h^j_p - h^j_{\p (p)}\big)+
O\left(|p-\p(p)|^2\right)\notag\\
={}&\sum_j \langle X_i(\p(p)), D\nu_j(\p(p))\cdot \mathbf{ex}^{-1}_{\p (p)} (p)\rangle h^j_{\p(p)}\nonumber\\
&\qquad\qquad+ O\left(|\vec \bT_F (p) - \vec \cM (\p (p))||p-\p(p)|+ |p-\p(p)|^2\right),
\end{align}
where we used elementary calculus to infer that $|\langle X^i(p), \nu_j(p)\rangle|\leq C
|p-\p(p)|$ and
\[
|h_p^j - h^j_{\p(p)} | \leq C\left(|\vec \bT_F (p) - \vec \cM (\p (p)|+ |p-\p(p)|\right).
\]
We only need that the constants $C$ appearing in the above inequalities are bounded by a geometric factor: in fact they enjoy explicit
bounds in terms of $\bmo^{\sfrac{1}{2}}$ which are at least linear, but such degree of precision is not needed. Finally recalling that
$p\in \supp (\bT_F)$, we can bound $|p - \p (p)| \leq |N (p)|$ and $|\vec \bT_F (p) - \vec \cM (\p (p))| \leq C |DN (\p (p))|$. We therefore conclude the estimate
\begin{equation*}
\langle X_i(p), h_p \rangle = \sum_j \langle X_i(\p(p)), D\nu_j(\p(p))\cdot \mathbf{ex}^{-1}_{\p (p)} (p)\rangle h^j_{\p(p)} + O \big(|N|^2 (\p (p)) +
|DN|^2 (\p (p))\big)\, .
\end{equation*}
We combine it with the expansion of the area functional
in \cite[Theorem~3.2]{DS2} to conclude the estimate on $I_2^i$. Recalling that $\p (F_i (x))= x$ we get
\begin{align*}
I_2^i & = \left|\int \langle X_i, h_p \rangle d\|\bT_F\| \right|
= \left| \sum_{i=1}^Q\int_\cM \langle Y, h_{F_i} \rangle \bJ F_i \right|
\allowdisplaybreaks\\
&\stackrel{\eqref{e:pezzo lineare}}{\leq} \left|\int_\cM \sum_{j=1}^l \sum_{i=1}^Q \langle Y (x), D\nu_j  (x) \cdot \mathbf{ex}^{-1}_{x} (F_i (x))\rangle h_{x}^j d\mathcal{H}^m (x)\right| + C\,\int_\cM \ph_r (|N|^2 + |DN|^2)
\end{align*}
Using the Taylor expansion for $\mathbf{ex}^{-1}_x$ at $x$ (and recalling that $F_i (x) - x = N_i (x)$) we conclude
\[
\Big|\sum_{i=1}^Q \mathbf{ex}_x^{-1} (F_i (x))\Big| \leq \left| {\rm d}\, \mathbf{ex}_x^{-1} (\etaa \circ N (x))\right| + O (|N|^2)
\leq C |\etaa \circ N (x)| + C |N|^2\, .
\]
Next consider that $|\langle Y, D\nu_j\cdot v \rangle | \leq C \varphi_r \|A_\Sigma\|_{C^0} |v| \leq C \varphi_r \, \bmo^{\sfrac12} |v|$
for every tangent vector $v$
and $|h^j_{x}| \leq  C\|A_\Sigma\|_{C^0} \leq \bmo^{\sfrac12}$. We thus conclude with the estimate
\[
I_2^i\leq  C\, \bmo^{} \int_\cM \ph_r\,|\etaa\circ N|+ C\,\int_\cM \ph_r (|N|^2 + |DN|^2) =: J_1+J_2\, .
\]
Clearly $J_1$ can be estimated as Err$_1^i$ and $J_2$ as Err$_2^i$, thus concluding the proof.

\section{Boundedness of the frequency}\label{s:freq_bound}
In this section we prove that the frequency function $\bI_j$ remains  bounded along
the different center manifolds corresponding to the intervals of flattening.
To simplify the notation, we set $p_j := \Phii_j (0)$ and write simply $\cB_\rho$
in place of $\cB_\rho (p_j)$ .

\begin{theorem}[Boundedness of the frequency functions]\label{t:boundedness}
Let $T$ be as in Assumption \ref{i:H'}.
If the intervals of flattening are $j_0 < \infty$, then there is $\rho>0$ such that
\begin{equation}\label{e:finita1}
\bH_{j_0} >0 \mbox{ on $]0, \rho[$} \quad \mbox{and} \quad \limsup_{r\to 0} \bI_{j_0} (r)< \infty\, .
\end{equation}
If the intervals of flattening are infinitely many, then there is a number $j_0\in \mathbb N$ { and a geometric constant $j_1\in \mathbb N$} such that
\begin{equation}\label{e:finita2}
\bH_j>0 \mbox{ on $]\frac{s_j}{t_j}, { 2^{-j_1} 3}[$ for all $j\geq j_0$}\, , \qquad 
\sup_{j\geq j_0} \sup_{r\in ]\frac{s_j}{t_j}, { 2^{-j_1} 3}[} \bI_j (r) <\infty\, ,
\end{equation}
{
\begin{equation}\label{e:finita3}
\sup \left\{\min \left\{\bI_j (r), \frac{r^2 \int_{\cB_r} |DN_j|^2}{\int_{\cB_r} |N_j|^2}\right\}\colon j \geq j_0 \mbox{ and } \max \left\{\frac{s_j}{t_j},\frac{3}{2^{j_1}}\right\} \leq r < 3\right\} < \infty\, 
\end{equation}
(in the latter inequality we understand $\bI_j (r) = \infty$ when $\bH_j (r)=0$)}
\end{theorem}

\begin{proof} 
Consider the first alternative. We claim that for every $r>0$ there is a radius $0<\rho < r$
such that $\bH (\rho) = \bH_{j_0} (\rho)>0$. Otherwise $N_{j_0}$ vanishes identically on some
 $\cB_r$. By \cite[Propositions~3.1 and 3.4]{DS4} and Proposition \ref{p:flattening}(iii)   
this is possible only if
no cube of the Whitney decomposition $\mathscr{W}^{(j_0)}$ intersects the projection of $\cB_r$
onto the plane $\pi$ (the reference plane for the construction of the center manifold). But then $T_{j_0}$ would coincide with
$Q\a{\cM}$ in $\B_{3r/4}$ and
$0$ would be a regular point of $T_{j_0}$ and, therefore, of $T$.

Next we claim that $\bH (r)>0$ for every $r\leq \rho$.
If not, let $r_0$ be the largest zero of $\bH$ which is smaller than $\rho$.
By Theorem \ref{t:frequency}, there is a constant $C$ such that
$\bI (r) \leq C (1+ \bI (\rho))$ for every $r \in ]r_0,\rho[$.
By letting $r\downarrow r_0$, we then conclude
\[
r_0 \bD (r_0) \leq C(1+ \bI (\rho)) \bH (r_0) = 0,
\]
that is, $N_j\vert_{\cB_{r_0}}\equiv0$ which we have already excluded.
Therefore, since $\bH >0$ on $]0, \rho[$, we can now apply
Theorem~\ref{t:frequency} to conclude \eqref{e:finita1}.

In the second case, we partition the extrema
$t_j$ of the intervals of  flattening into two different classes: the class $(A)$ when $t_j = s_{j-1}$ and the class $(B)$ when $t_j<s_{j-1}$.
If $t_j$ belongs to $(A)$, set $r:= \frac{s_{j-1}}{t_{j-1}}$. Let $L \in \sW^{(j-1)}$ be
a cube of the Whitney decomposition such that $c_s\, r \leq \ell(L)$ and 
$L\cap \bar B_r(0, \pi) \neq \emptyset$.
We are in the position to apply \cite[Proposition~3.7]{DS4} for the comparison of
two center manifolds: there exists a constant $\bar c_s >0$ such that
\[
\int_{\B_2\cap \cM_j} |N_j|^2 
\geq \bar{c}_s \bmo^j:= \bar c_s\max\big\{\bE (T_j, \B_{6\sqrt{m}}), \mathbf{c} (\Sigma_j)^2\big\},
\]
which obviously gives  { $\int_{\cB_3} |N_j|^2 \geq c \bmo^j$.
By \cite[(2.7)]{DS4} (or alternatively by \eqref{e:rough}), we then conclude
\begin{equation}\label{e:anello_palla_2}
\int_{\cB_3} |N_j|^2 \geq \bar c \int_{\cB_3} |DN_j|^2\, ,
\end{equation}
where $\bar c$ is a positive geometric constant
By the H\"older inequality and Sobolev embedding (cf. \cite[Proposition 2.11]{DS1}), there are geometric constants $C_0$ and 
$\bar\alpha = m (1-\frac{2}{q})>0$ such that
\begin{align}\label{e:anello_palla_3}
\int_{\cB_{\frac{3}{2^{J}}}} |N_j|^2 \leq & \left(\cH^m \left(\cB_{\frac{3}{2^{J}}}\right)\right)^{1-\sfrac{2}{q}}
\left(\int_{{\frac{3}{2^{J}}}} |N_j|^q\right)^{\sfrac{2}{q}}
\leq  C_0 2^{- J \bar\alpha} \int_{\cB_3} |N_j|^2 + C_0 2^{- J \bar \alpha} \int_{\cB_3} |DN_j|^2\nonumber\\
\leq & C_0 2^{-J\bar\alpha} \bar{c}^{-1} \int_{\cB_3} |N_j|^2\qquad\qquad\qquad
\mbox{for any $J\in \mathbb N$}\, 
\end{align}
(in the above we can set $q= 2^\star$ when $m\geq 3$ and choose any $q<\infty$ larger than $2$ for $m=2$; note also that since the curvature of the manifold $\cM_j$ is bounded by $\bmo^j$, we can assume that $\cH^m (\cB_\rho)$ is comparable to the $m$-dimensional volume of the corresponding euclidean ball for every $\rho<3$). 
If we choose $J = j_1$ for a large enough $j_1$ (depending only upon $\bar c$, $\alpha$ and $C_0$) we achieve
\begin{equation}\label{e:anello_palla_4}
\int_{\cB_3\setminus \cB_{\frac{3}{2^{j_1}}}} |N_j|^2 \geq \frac{1}{2} \int_{\cB_3} |N_j|^2 \geq \frac{\bar c}{2} \int_{\cB_3} |DN_j|^2\, .
\end{equation}
In turn we conclude the existence one annulus $\cA (k(j)) := \cB_{3/(2^{k(j)})}\setminus \cB_{3/(2^{k(j)+1})}$ with
\begin{equation}\label{e:anello_palla_5}
\int_{\cA (k(j))} |N_j|^2 \geq \frac{\bar c}{2j_1} \int_{\cB_3} |DN_j|^2\qquad \mbox{and}\qquad k(j) \leq j_1\, .
\end{equation}
$\bH_{N_j} (k(j))$ is bounded from below by the integral on the left hand side of \eqref{e:anello_palla_5}, whereas the right hand side bounds
$\bD_{N_j} (2^{-k(j)} 3)$ from above. Thus $\bI_{N_j} (2^{-k(j)} 3)$ is smaller than a constant which depends upon $\bar{c}$ and $j_1$.
Arguing as in the first alternative, we can apply Theorem \ref{t:frequency} to conclude the positivity of $\bH_{N_j}$ and to gain a uniform upper bound for $\bI_{N_j}$ on the interval $]\frac{s_j}{t_j}, 2^{-k(j)} 3[$: since the latter contains $]\frac{s_j}{t_j}, 2^{j_1} 3[$, we conclude the validity of \eqref{e:finita2} (if one or both the intervals are trivial, namely $\frac{s_j}{t_j}$ is larger than the right endpoint, then there is nothing to prove). On the other hand for every $r\in [2^{-k(j)} 3, 3[$, by \eqref{e:anello_palla_5} we certainly have
\[
\int_{\cB_r} |N_j|^2 \geq \frac{\bar c}{2j_1} \int_{\cB_3} |DN_j|^2
\]
from which \eqref{e:finita3} readily follows}.

In the case $t_j$ belongs to the class $(B)$, 
then, by construction there is $\eta_j \in ]0,1[$ such that
$\bE ((\iota_{0, t_j})_\sharp T, \B_{6\sqrt{m}(1+\eta_j)}) > \eps_3^2$. Up to extraction of a subsequence, we can assume that $(\iota_{0, t_j})_\sharp T$ converges to a cone $S$: the convergence is strong enough to conclude that the excess
of the cone is the limit of the excesses of the sequence. Moreover (since $S$ is a cone), the excess $\bE (S, \bB_r)$ is independent of $r$. We then conclude 
\[
\eps_3^2 \leq  \liminf_{j\to\infty, j\in (B)} \bE (T_j, \B_3)\, .
\]
Thus, by Lemma \ref{l:flatness} below, we conclude $\liminf_{j\to\infty, j\in (B)} \bH_{N_j} (3) > 0$. Since $\bD_{N_j}(3) \leq C \bmo^j\leq C\eps_3^2$, we achieve that $\limsup_{j\to\infty, j\in (B)} \bI_{N_j} (3)<+\infty$, and conclude as before.
\end{proof}

\begin{lemma}\label{l:flatness}
Assume the intervals of flattening are infinitely many and
$r_j \in ]\frac{s_j}{t_j},3[$ is a subsequence
(not relabeled) with $\lim_j \|N_j\|_{L^2(\cB_{r_j} \setminus \cB_{r_j/2})} = 0$. If $\eps_3$ is sufficiently small,
then, $\bE(T_{j},\B_{r_j}) \to 0$. 
\end{lemma}

\begin{proof} 
Note that, if $r_j \to 0$, then necessarily $\bE(T_{j},\B_{r_j}) \to 0$
by Proposition~\ref{p:flattening}(iv).
Therefore, up to a subsequence, we can assume the existence of $c>0$ such that
\begin{equation}\label{e:nontrivial}
r_j \geq c \quad \text{and} \quad \bE (T_{j}, \B_{6\sqrt{m}})\geq c.
\end{equation}
After the extraction of a further subsequence, we can assume the existence of $r$ such that
\begin{equation}\label{e:(A1)_senza_scale}
\int_{\cB_{r} \setminus \cB_{\frac{3 r}{4}}} |N_{j}|^2 \to 0,
\end{equation}
and the existence of an area-minimizing cone $S$ such that $(\iota_{0, t_j})_\sharp T\to S$.
Note that, by \eqref{e:nontrivial}, $S$ is not a multiplicity $Q$ flat $m$-plane.
Consider the orthogonal projection $\q_j:\R^{m+n} \to \pi_j$, where $\pi_j$ is the $m$-dimensional plane of the construction of the center manifold $\cM_j$. Assuming $\eps_3$ is sufficiently small,
we have $U_j := B_{\frac{15}{16} r} (\pi)\setminus B_{\frac{13}{16} r}(\pi) \subset \q_j (\cB_r\setminus \cB_{\frac{3}{4}r})$.
{ Consider the Whitney decomposition $\sW^{(j)}$ leading to the
construction of $\cM_j$: if no cube of the decomposition iintersects $U^j$, then $N_j$ vanishes identically on it.} Otherwise, set
\[
d_j := \max \big\{\ell (J): J\in \sW^{(j)} \quad\text{and}\quad J\cap U_j\neq \emptyset\big\}.
\]
Let $J_j\in \sW^{(j)}$ be such that $U_j\cap J_j \neq \emptyset$ and $d_j = \ell (J_j)$. If the stopping condition for
$J_j$ is either (HT) or (EX), recalling that $\ell (J_j) \leq c_s r$,
we choose a ball $B^j\subset U_j$ of radius $\frac{d_j}{2}$ and at distance at most $\sqrt{m} d_j$ from $J_j$.
If the stopping condition for $J_j$ is (NN), $J_j$ is in the domain
of influence of $K_j\in \mathscr{W}^{(j)}_e$. By Proposition \ref{p:flattening} we can then
choose a ball $B^j\subset U_j$ of radius $\frac{\ell (K_j)}{8}$ at distance at most $3\sqrt{m} \ell (K_j)$
from $K_j$.
If the stopping condition is (HT), we then have by \cite[Proposition~3.1]{DS4}
\[
\int_{\cB_{r} \setminus \cB_{\frac{3 r}{4}}} |N_{j}|^2\geq \int_{\Phii_j (B^j)} |N_j|^2 \geq c \left(\bmo^j\right)^{\frac{1}{m}} d_j^{m+2+2\beta_2}\, .
\]
If the stopping condition is either (NN) or (EX), { by \cite[Proposition~3.1]{DS4} and \cite[Proposition 3.4]{DS4}} we have 
\begin{align}
\int_{\cB_{r} \setminus \cB_{\frac{3 r}{4}}} |N_{j}|^2  \geq \int_{\Phii_j (B^j)} |N_j|^2 
\geq c d_j^2 \int_{\Phii_j (B^j)} |DN_j|^2
\geq c \bmo^j d_j^{m + 4 - 2\delta_2}.
\end{align}
In both cases we conclude that $d_j\to 0$.

By \cite[Corollary~2.2]{DS4}, $\supp (T_j)\cap \Phii_j (U_j)$ is contained in a $d_j$-tubular neighborhood of $\cM_j$, which we denote by $\hat{\bU}_j$. Moreover, again assuming that $\eps_3$ is sufficiently small, we can assume
$\B_t\setminus \B_s\cap \cM_j \subset \Phii_j (U_j)$ for some appropriate choice of $s<t$, independent of $j$.
Finally, by \cite[Theorem~1.17]{DS4} we can assume that (up to subsequences) $\cM_j$ converges to $\cM$ in $C^3$. 
We thus conclude that { $S\res (\B_t\setminus \bar \B_s)$ is supported in $\cM\cap (\B_t\setminus \bar\B_s)$  and, hence,
by the constancy theorem, $S\res (\B_t\setminus\bar \B_s) = Q_0 \a{\cM \cap (\B_t\setminus\bar \B_s)}$ for some integer $Q_0$.
Observe also that, if
$\p_j:\hat{ \bU}_j \to \cM_j$ is the least distance projection onto $\cM_j$, by  \cite[Theorem~2.4]{DS4}
we also have $(\p_j)_\sharp (T_j \res (\B_t\setminus\bar \B_s)) = Q \a{\cM_j \cap (\B_t\setminus\bar \B_s)}$.} We therefore conclude that $Q_0=Q$.
%
Since $S$ is a cone without boundary, $\partial (S\res\B_t) = Q \a{\cM\cap \de\B_t}$,
i.e.~$S\res \bB_t = Q\a{0} \cone \a{\cM\cap \de\B_t}$.
By Allard's regularity theorem { (which can be applied because $\Theta (S, 0) = \lim_j \Theta (T_j, 0) = Q$)}, $S$ is regular in a neighborhood of $0$ and, therefore, it is an $m$-plane with multiplicity $Q$, which gives the desired contradiction.
\end{proof}

A corollary of Theorem \ref{t:boundedness} is the following.

\begin{corollary}[Reverse Sobolev]\label{c:rev_Sob}
Let $T$ be as in Assumption \ref{i:H'}.
Then, there exists a constant $C>0$ which {\em depends on $T$ but not on $j$} such that, for every $j$
and for every $r \in ]\frac{s_j}{t_j}, 1 ]$, there is $s\in ]\frac{3}{2}r, 3r]$
such that
\begin{equation}\label{e:rev_Sob}
\int_{\cB_{s}(\Phii_j(0))} |D{N}_j|^2 \leq \frac{C}{r^2} \int_{\cB_{s} (\Phii_j(0))} |{N}_j|^2\, .  
\end{equation}
\end{corollary}

\begin{proof} { If the second alternative in Theorem \ref{t:boundedness} holds, if $r\geq 2^{-j_1} 3$ and if $\bI_j (3r)$ is larger than the ratio 
\[
\frac{(3r)^2 \int_{\cB_{3r} (\Phii_j(0))} |DN_j|^2}{\int_{\cB_{3r} (\Phii_j(0))} |N_j|^2}\, ,
\]
then the claim follows from \eqref{e:finita3}. 
Therefore, without loss of generality, we can assume that $\bI_j (3r)$ is bounded by a constant $C^\star$, which depends on $T$ but not on $j$.}

We start observing that, by the Coarea Formula,
\[
\bH_j (3r) = \int_{\cB_{3r} (\Phii_j (0)) \setminus \cB_{3r /2} (\Phii_j (0))} 2 \frac{|N_j|^2}{d(p)} 
= 2 \int_{3r/2}^{3r} \frac{1}{t} \int_{\partial \cB_t (\Phii_j (0))} |N_j|^2\, dt\, ,
\]
whereas, using Fubini,
\[
\int_{\frac32 r}^{3r}\int_{\cB_t(\Phii_j(0))} |D{N}_j|^2\, dt
= \int_{\cM_j} |DN_j|^2 (x) \int_{3 r/2}^{3r} {\bf 1}_{]|x|, \infty[} (t)\, dt\, d\cH^m (x)
= \frac{3}{2} r \bD_j (3r)\, .
\]
{ Since we are assuming that $\bI_j (3r)\leq C^\star$}
\begin{align*}
\int_{\frac32 r}^{3r} dt \int_{\cB_t(\Phii_j(0))} |D{N}_j|^2 = \frac32 r\,\bD_{j} (3r) \leq
{ C^\star} \,\bH_{j} (3r)
= { C^\star} \int_{\frac32 r}^{3r} dt
\frac{1}{t} \int_{\partial \cB_t(\Phii_j(0))} |{N}_j|^2\, .  
\end{align*}
Therefore, there must be $s\in [\frac{3}{2}r, 3r]$ such that
\begin{equation}\label{e:un_raggio_buono}
\int_{\cB_{s}(\Phii_j(0))} |D{N}_j|^2 \leq { \frac{C^\star}s} \int_{\partial \cB_{s}(\Phii_j(0))} |{N}_j|^2\, .
\end{equation}
Fix now { any $\sigma \in ]s/2, s[$} and any point $x\in \partial \cB_s (\Phi_j (0))$. Consider the geodesic line $\gamma$
passing through $x$ and $\Phi_j (0)$ and let $\hat{\gamma}$ be the arc on $\gamma$ having one endpoint $\bar{x}$ in $\partial \cB_\sigma (\Phi_j (0))$ and one endpoint equal to $x$. Using \cite[Proposition 2.1(b)]{DS1} and the fundamental theorem of calculus, we easily conclude
\[
|N_j (x)|^2 \leq |N_j (\bar x)|^2 + 2 \int_{\hat{\gamma}} |DN_j| |N_j|\, .
\]
Integrating this inequality in $x$ and recalling that $\sigma > s/2$ we then easily conclude
\[
\int_{\partial \cB_s (\Phii_j (0))} |N_j|^2 \leq C \int_{\partial \cB_\sigma (\Phii_j (0))} |N_j|^2 + C \int_{{ \cB_s (\Phii_j (0))\setminus \cB_{s/2} (\Phii_j (0))}} |N_j||DN_j|\, ,
\]
{ where the constant $C$ depends only on the curvature of $\cM_j$, which is bounded independently of $j$.}
We further { integrate in $\sigma$ between $s/2$ and $s$} to achieve
\begin{align}\label{e:young}
{ \frac{s}{2}} \int_{\partial \cB_{s}(\Phii_j(0))} |{N}_j|^2 & \leq C
 \int_{{ \cB_{s}(\Phii_j(0))\setminus \cB_{s/2} (\Phii_j (0))}} \left(|{N}_j|^2 + s\,|{N}_j| |D{N}_j|\right)\notag\\
&\leq { \frac{s^2}{4 C^\star}} \int_{\cB_{s}(\Phii_j(0))} |D{N}_j|^2 + { \bar C} \int_{\cB_{s}(\Phii_j(0))} |{N}_j|^2\, ,
\end{align}
{ where $C^\star$ is the constant in \eqref{e:un_raggio_buono} and the constant $\bar{C}$ depends on the curvature of $\cM_j$ and on $C^\star$.} Combining \eqref{e:young} with  \eqref{e:un_raggio_buono}
we easily conclude \eqref{e:rev_Sob}.
\end{proof}

\section{Final blow-up sequence and capacitary argument}\label{s:setting+capacitario}

\subsection{Blow-up maps}\label{ss:blowup}
Let $T$ be a current as in the Assumption \ref{i:H'}.
By Proposition~\ref{p:flattening} we can assume that for each radius $r_k$
there is an interval of flattening $I_{j(k)} =]s_{j(k)}, t_{j(k)}]$ containing $r_k$.
We define next the sequence of ``blow-up maps'' which will lead to the proof
of Almgren's partial regularity result Theorem~\ref{t:main}.
To this aim, for $k$ large enough, we define $\bar{s}_k$ so that the radius $\frac{\bar s_k}{t_{j(k)}} \in \big]\frac32 \frac{r_k}{t_{j(k)}}, 3\frac{r_k}{t_{j(k)}}[$ is the radius
provided in Corollary~\ref{c:rev_Sob} applied to $r = \frac{r_k}{t_{j(k)}}$. We then set
$\bar{r}_k := \frac{2\bar s_k}{3t_{j(k)}}$
and rescale and translate currents and maps accordingly:
\begin{itemize}
\item[(BU1)] $\bar T_k = (\iota_{0,\bar r_k})_\sharp T_{j(k)} = ((\iota_{0,\bar r_k t_{j(k)}})_\sharp T) \res \bB_{6\sqrt{m}/\bar{r}_k} $, $\bar\Sigma_k = \iota_{0,\bar r_k}(\Sigma_{j(k)})$\\ and $\bar\cM_k := \iota_{0, \bar{r}_k} (\cM_{j(k)})$;
\item[(BU2)] $\bar{N}_k : \bar\cM_k \to \R^{m+n}$ are the rescaled $\bar\cM_k$-normal approximations
given by 
\begin{equation}\label{e:bar_N_k}
\bar{N}_k (p) = \frac{1}{\bar{r}_k} N_{j(k)} (\bar{r}_k p).
\end{equation}
\end{itemize}
Since by assumption $T_{0} \Sigma =\R^{m+\bar{n}}\times \{0\}$,
the ambient manifolds $\bar{\Sigma}_k$ converge to 
$\R^{m+\bar{n}}\times \{0\}$ locally in $C^{3,\eps_0}$ (more precisely to a ``large portion'' of 
$\R^{m+\bar{n}}\times \{0\}$, because $\bB_{6\sqrt{m}} \subset \bB_{6\sqrt{m}/\bar{r}_k}$).
Moreover, since $\frac12 < \frac{r_k}{\bar r_k t_{j(k)}} < 1$, it follows from
Proposition~\ref{p:seq} that
\[
\bE(\bar T_k, \bB_{\frac12}) \leq C\bE(T, \bB_{r_k}) \to 0.
\]
By the standard regularity theory of area minimizing currents and
Assumption~\ref{i:H'}, this implies that
$\bar T_k$ locally converge (and supports converge locally in the Hausdorff sense) to (a large portion of) a minimizing tangent cone
which is an $m$-plane with multiplicity $Q$ contained in $\R^{m+\bar n}\times\{0\}$.
Without loss of generality, we can assume that $\bar{T}_k$ locally converge to $Q\a{\pi_0}$.
Moreover, from Proposition~\ref{p:seq} it follows that
\begin{equation}\label{e:sing grande}
\cH^{m-2+\alpha}_\infty (\rD_Q (\bar{T}_k) \cap \bB_1)
\geq C_0 r_k^{- (m-2+\alpha)}\cH^{m-2+\alpha}_\infty (\rD_Q ({T}) \cap \bB_{r_k})\geq \eta>0\, ,
\end{equation}
where $C_0$ is a geometric constant.

In the next lemma, we show that the rescaled center manifolds $\bar \cM_k$ converge locally to
the flat $m$-plane $\pi_0$, thus leading to the following natural definition for the 
blow-up maps $N^b_k : B_{3}\subset\R^m \to \Iq (\R^{m+n})$:
\begin{equation}\label{e:sospirata_successione}
N^b_k (x) := \bh_k^{-1} \bar{N}_k (\be_k (x))\, ,
\end{equation}
where $\bh_k:=\|\bar N_k\|_{L^2(\cB_{\frac32})}$ and
$\be_k: B_3\subset\R^m\simeq T_{\bar{p}_k} \bar\cM_k\to \bar\cM_k$ denotes the exponential map
at $\bar{p}_k = \Phii_{j(k)} (0)/\bar{r}_k$ ({ here and in what follows we assume, w.l.o.g.,
to have applied a suitable rotation to each $\bar{T}_k$ so that the tangent plane $T_{\bar{p}_k} \bar{\mathcal{M}}_k$ coincides with
$\mathbb R^m\times \{0\}$}).



\begin{lemma}[Vanishing lemma]\label{l:vanishing}
Under the Assumption~\ref{i:H'},  the following hold:
\begin{itemize}
\item[(i)] we can assume, without loss of generality, $\bar{r}_k \bmo^{j(k)}\to 0$; 
\item[(ii)] the rescaled center manifolds $\bar{\cM}_k$ converge
(up to subsequences) to
$\R^m\times \{0\}$ in $C^{3,\kappa/2} (\bB_4)$ and
the maps $\be_k$ converge in $C^{2, \kappa/2}$ to the identity map
${\rm id}: B_3 \to B_3$;
\item[(iii)] there exists a constant $C>0$, depending only $T$, such that, for every $k$,
\begin{equation}\label{e:rev_Sob2}
\int_{B_{\frac32}} |D{N}^b_k|^2 \leq C. 
\end{equation}
\end{itemize}
\end{lemma}

\begin{proof}
To show (i), note that, if $\liminf_k \bar{r}_k > 0$, we can extract a further subsequence and assume that $\lim_k \bar{r}_k >0$. Observe that then $\bar r := \limsup_k \frac{t_{j(k)}}{r_k} <\infty$. Since $r_k\downarrow 0$, we necessarily conclude that $t_{j(k)}\downarrow 0$ and hence $\mathbf{c} (\Sigma_{j(k)})\to 0$.
Moreover 
$\bE (T, \bB_{6 \sqrt{m} t_{j(k)}}) \leq C (\bar r) \bE (\bar T_k, \bB_{6\sqrt{m} \bar r_k^{-1}}) \to 0$
because $\bar T_k$ converges to $Q\a{\pi_0}$. 
We conclude $\bar{r}_k \bmo^{j(k)} \to 0$. On the other hand if $\lim_k \bar{r}_k = 0$ then (i) follows trivially from
the fact that $\bmo^j$ is a bounded sequence.

Next, using $\bar{r}_k\bmo^{j(k)}\to 0$ and the estimate of \cite[Theorem~1.17]{DS4},
it follows easily that $\bar \cM_k - \bar p_k$
converge (up to subsequences) to a plane in $C^{3,\kappa/2} (\bB_4)$. 
By Proposition~\ref{p:flattening} (v) we deduce easily that such plane is in fact $\pi_0$. Since $0$ belongs to the support of $T_{j(k)}$
we conclude for the same reason that $\bar \cM_k$ is converging to $\pi_0$ as well.
Therefore, by Proposition~\ref{p:exp} the maps $\be_k$ converge to
the identity in $C^{2,\kappa/2}$ (indeed, by standard arguments
they must converge to the exponential map on the -- totally geodesic! --
submanifold $\R^m\times \{0\}$).
Finally, (iii) is a simple consequence of Corollary~\ref{c:rev_Sob}.
\end{proof}

The main result about the blow-up maps $N^b_k$ is the following.

\begin{theorem}[Final blow-up]\label{t:sospirato_blowup}
Up to subsequences, the maps $N^b_k$ converge strongly in $L^2(B_\frac32)$ to a function
$N^b_\infty: B_{\frac32} \to \Iq (\{0\}\times \R^{\bar{n}}\times \{0\})$ which
is Dir-minimizing in $B_t$ for every $t\in ]\frac{5}{3}, \frac{3}{2}[$ and satisfies
$\|N^b_\infty\|_{L^2(B_\frac32)}=1$ and $\etaa\circ N^b_\infty \equiv 0$.
\end{theorem}

We postpone the proof of Theorem~\ref{t:sospirato_blowup} to the next section
and show next Theorem~\ref{t:main}. 

\subsection{Proof of Theorem \ref{t:main}: capacitary argument}
Let $N^b_\infty$ be as in Theorem~\ref{t:sospirato_blowup} and
\[
\Upsilon := \left\{x\in \bar B_{1} : N^b_\infty(x) = Q\a{0}\right\}.
\]
Since $\etaa \circ N^b_\infty\equiv 0$ and
$\| N^b_\infty\|_{L^2(B_{\sfrac{3}{2}})}= 1$, from the regularity of Dir-minimizing
$Q$-valued functions (cf.~\cite[Proposition 3.22]{DS1}),
we know that
$\cH^{m-2+\alpha}_\infty (\Upsilon) = 0$.
We show in the next three steps that this contradicts Assumption~\ref{assurdo}.

\medskip

\textit{Step 1.} 
We cover $\Upsilon$ by balls $\{\bB_{\sigma_i} (x_i)\}$ in such a way that
$\sum_i \omega_{m-2+\alpha} (4\sigma_i)^{m-2+\alpha} \leq 
\frac{\eta}{2}$, where $\eta$ is the constant in \eqref{e:sing grande}.
By the compactness of $\Upsilon$, such a covering can be chosen finite.
We can therefore choose a $\bar{\sigma}>0$ so that
the $5\bar\sigma$-neighborhood of $\Upsilon$ is covered by $\{\bB_{\sigma_i} (x_i)\}$.
Denote by $\Lambda_{k}$ the set of multiplicity $Q$ points of $\bar T_k$ far away from the singular set
$\Upsilon$:
\[
\Lambda_{k} := \{ p \in D_Q(\bar T_k) \cap \bB_1 : \dist(p, \Upsilon)>4 \bar{\sigma} \}.
\]
Clearly, $\cH^{m-2+\alpha}_\infty (\Lambda_k)
\geq \frac{\eta}{2}$. Let $\mathbf{V}$ denote the neighborhood of $\Upsilon$ of size $2\bar{\sigma}$.
By the H\"older continuity of $\D$-minimizing
functions (cf.~\cite[Theorem~2.9]{DS1}) there is a positive constant $1>\vartheta>0$ such that
$|N^b_\infty (x)|^2\geq 2 \vartheta$ for every $x\not\in \mathbf{V}$. We next introduce a parameter
$\sigma > 0$ whose choice will be specified only at the very end: throughout the rest of the proof it
will only required to be sufficiently small. In particular, $\sigma < \bar{\sigma}$ will surely imply that
\[
\mint_{B_{2\sigma} (x)} |N^b_\infty|^2 \geq 2 \,\vartheta \qquad \mbox{$\forall\;x\in B_{\frac54}$
with  $\dist (x, \Upsilon) \geq 4 \bar{\sigma}$.}
\]
Therefore, from Theorem~\ref{t:sospirato_blowup} we infer that, for sufficiently large $k$'s,
\begin{equation}\label{e:dall'alto}
\mint_{\cB_{2\sigma} (x)} \cG (\bar{N}_k, Q \a{\etaa\circ \bar{N}_k})^2 \geq \vartheta \bh_k^2
\qquad \mbox{$\forall\;x\in \Gamma_k:= \p_{\bar\cM_k} (\Lambda_k)$.}
\end{equation}

\medskip

\textit{Step 2.} For every $p\in \Lambda_k$, consider $\bar{z}_k (p) = \p_{\pi_k} (p)$ (where $\pi_k$ is the reference plane for the center manifold
related to $T_{j(k)}$) and 
$\bar{x}_k (p) := (\bar{z}_k (p),  \bar{r}_k^{-1} \phii_{j(k)} (\bar{r}_k z_k (p)))$. Observe that $\bar{x}_k (p)\in \bar\cM_k$. We next claim the existence of 
a suitably chosen geometric constant $1>c_0>0$ (in particular, independent of $\sigma$) such that, when $k$ is large enough,
for each $p\in \Lambda_k$ there is a radius $\varrho_p \leq 2\sigma$ with the following properties:
\begin{gather}
\frac{c_0\, \vartheta}{\sigma^{\alpha}}  \bh_k^2\leq \frac1{\varrho_p^{m-2+\alpha}}
\int_{\cB_{\varrho_p} (\bar{x}_k (p))} |D\bar{N}_k|^2,
\label{e:dal_basso_finale}\\
\cB_{\varrho_p} (\bar{x}_k (p)) \subset \bB_{4\varrho_p} (p)\label{e:inside}\, .
\end{gather}
In order to show this claim, fix such a point $p$, consider the point $q_k :=\bar{r}_k  p$, $z_k := \bar{r}_k\bar{z}_k (p)$ and
$x_k = \bar{r}_k \bar{x}_k (p) =(z_k, \phii_{j(k)} (z_k))$. Observe that $q_k \in D_Q (T_{j(k)})$. 
By \cite[Proposition~3.1]{DS4}, $z_k$ cannot belong to some $L\in \sW^{(j(k))}_h$ (otherwise $\bB_{16r_L} (p_L)$ would contain 
a multiplicity $Q$ point of $T_{j(k)}$, contradicting statement (S1) in \cite[Proposition~3.1]{DS4}). We thus distinguish two possibilities:
\begin{itemize}
\item[(Exc)] either $z_k$ belongs to some $L_k \in \sW^{(j(k))}_e \cup \sW^{(j(k))}_n$;
\item[(Con)] or it belongs to the set $\Gamma_{j(k)}$. 
\end{itemize}

\medskip

{\bf Case (Exc).}
Observe that if $L_k \in \sW^{(j(k))}_n$, by Proposition~\ref{p:flattening} (iii), there
exists a cube $H_k \in \sW^{(j(k))}_e$ such that $L_k$ belongs to the domain of influence
of $H_k$ and $\textup{sep}(B_{\bar{r}_k}, H_k) \leq 3 \bar{r}_k/16$. Thus $H_k$ intersects $B_{19 \bar{r}_k/16} (0, \pi)$.

We wish now to apply \cite[Proposition~3.5]{DS4} { with $s$ in there equal to $\bar{r}_k$ and $T$ in there equal to $T_{j(k)}$}: the aim is to infer
\begin{equation}\label{e:vanishing}
\bar{\ell}_k := 
\bar{r}_k^{-1} \sup \big\{\ell (L): L\in \sW^{(j(k))}_e \;\mbox{and}\; L\cap B_{19\bar{r}_k/16} (0, \pi)\neq \emptyset\big\} = o(1).
\end{equation}
{ Observe first that, taking into account the inequality $1\leq \bar{r}_k t_{j(k)}/r_k \leq 2$, a simple scaling argument gives \begin{align*}
&\mathcal{H}^{m-2+\alpha} (D_Q (T_{j(k)}), \bB_{\bar{r}_k}) \geq \left(\textstyle{\frac{r_k}{t_{j(k)}}}\right)^{m-2+\alpha}
\mathcal{H}^{m-2+\alpha} (D_Q (T_{0, r_k})\cap \bB_{t_{j(k)} \bar{r}_k/r_k})\\
\geq &  \left(\textstyle{\frac{r_k}{t_{j(k)}}}\right)^{m-2+\alpha} 
\mathcal{H}^{m-2+\alpha} (D_Q (T_{0, r_k})\cap \bB_1)\;\stackrel{\eqref{e:seq2}}{\geq}\; \eta\left(\textstyle{\frac{\bar{r}_k}{2}}\right)^{m-2+\alpha}\, , 
\end{align*}
which verifies \cite[(3.4)]{DS4}. We next need to verify \cite[(3.3)]{DS4} and consider therefore $L\in \sW^{(j)}$ which intersects $B_{3\bar{r}_k} (0, \pi)$. Since $\bar{r}_k >s_{j(k)}/t_{j(k)}$, by (Go) we have $\ell (L) <3 c_s \bar{r}_k \leq \bar{r}_k$. Now, for any fixed $\hat{\alpha}>0$ we can apply \cite[Proposition~3.5]{DS4} provided $\min \{\bar{r}_k, \bmo^{j(k)}\}$ is small enough, which is the case for $k$ large enough by Lemma \ref{l:vanishing}(i). Thus \cite[Proposition~3.5]{DS4} implies $\limsup_k \bar{\ell}_k \leq \hat{\alpha}$ and the arbitrariness of the latter parameter implies \eqref{e:vanishing}.}

For $k$ large enough, we can then apply \cite[Proposition~3.6]{DS4}
with $\eta_2= \frac{\vartheta}{4}$ (in particular this condition on how large $k$ must be is {\em independent} of the point $p$) . The Proposition
will be applied to $L_k$, if $L_k \in \sW^{(j(k))}_e$, or to $H_k$ above, if $L_k \in \sW^{(j(k))}_n$.
{ We thus
set
\[
 J_k=\begin{cases} H_k&\hbox{if $L_k\in\sW_n^{(j(k))}$,}\cr
        L_k&\hbox{if $L_k\in\sW_e^{(j(k))}$}
\end{cases}
\]
and conclude the existence of a constant $\bar s <1$ such that
\begin{equation*}\label{e:sotto_dal_cm}
\mint_{\cB_{\bar{s} \ell (J_k)} (x_k)} \cG (N_{j(k)}, Q \a{\etaa\circ N_{j(k)}})^2 
\leq \frac{\vartheta}{4 \omega_m \ell (J_k)^{m-2}} \int_{\cB_{\ell (J_k)} (x_k)} |DN_{j(k)}|^2\, .
\end{equation*}
By \eqref{e:vanishing} have, provided $k$ is large enough,
$t (p) := \frac{\ell(L)}{\bar r_k} \leq \bar\ell_k \leq \sigma$.
Therefore, rescaling to $\bar\cM_k$,
there exists $t (p) \leq \bar \ell_k$ such that
\begin{equation}\label{e:sotto_dal_cm_2}
\mint_{\cB_{\bar{s} t (p)} (\bar{x}_k (p))} \cG (\bar{N}_k, Q \a{\etaa\circ \bar{N}_k})^2 \leq 
\frac{\vartheta}{4 \omega_m t (p)^{m-2}} \int_{\cB_{t (p)} (\bar{x}_k (p))} |D\bar{N}_k|^2\, .
\end{equation}}
Moreover, from Proposition~\ref{p:flattening} (v) and Lemma~\ref{l:vanishing},
for $k$ large enough, we get
\begin{equation}\label{e:vicinanza_2}
|p- \bar{x}_k (p)|\leq C (\bmo^{j(k)})^{\frac{1}{2m}} \bar{r}_k^{\beta_2} t (p) < \bar s\, t (p).
\end{equation}

\medskip

{\bf Case (Con).} In case $q_k$ belongs to the contact set $\Phii_{j(k)} (\mathbf{\Gamma}_{j(k)})$, then $p=x_k (p)$ and $N_{j(k)} (x_k (p))=
Q\a{0}$. Therefore 
\[
\lim_{t\downarrow 0} 
\mint_{\cB_t (\bar{x}_k (p))} \cG (\bar{N}_k, Q \a{\etaa\circ \bar{N}_k})^2 = 0
\]
and we choose $t (p) < \sigma$ such that 
\begin{equation}\label{e:ancora_un_caso}
 \mint_{\cB_{\bar{s} t (p)} (\bar{x}_k (p))} \cG (\bar{N}_k, Q \a{\etaa\circ \bar{N}_k})^2
\leq \frac{\vartheta}{4} \bh_k^2\, . 
\end{equation}
Observe also that \eqref{e:vicinanza_2} holds trivially.

\medskip

Having chosen $t(p)$ in both cases, we next show
the existence of $\varrho_p \in ]\bar s\, t (p), 2\sigma[$
such that \eqref{e:dal_basso_finale} holds. Observe that \eqref{e:inside} will be an obvious consequence of \eqref{e:vicinanza_2}.
Notice that if
\begin{equation}\label{e:alto_rho_x}
\frac{1}{\omega_m t (p)^{m-2}} \int_{\cB_{t (p)} (\bar{x}_k (p))} |D \bar N_k|^2 \geq \bh_k^2\,,
\end{equation}
then \eqref{e:dal_basso_finale} follows with $\varrho_p = t (p)$.
If \eqref{e:alto_rho_x} does not hold, then
\begin{equation}\label{e:L2_basso}
 \mint_{\cB_{\bar{s} t (p)} (\bar{x}_k (p))} \cG (\bar{N}_k, Q \a{\etaa\circ \bar{N}_k})^2
\leq \frac{\vartheta}{4} \bh_k^2\, .
\end{equation}
Indeed we can use
\eqref{e:sotto_dal_cm_2} in the case (Exc) (in the case (Con) we have already shown it: see \eqref{e:ancora_un_caso}).

We now argue by contradiction to infer the existence of $\varrho_p \in ]\bar s t (p), 2\sigma]$
such that \eqref{e:dal_basso_finale} holds. Indeed, if this were not the case, 
{ we set for simplicity $f:= \cG (\bar{N}_k, Q\a{\etaa\circ \bar{N}_k})$,
\[
\bar{f}_r := \mint_{\cB_r (\bar{x}_k(p))} f
\]
and, letting $j$ be the smallest integer such that
$2^{-j} \sigma \leq \bar{s} t (p)$, we can estimate as follows:
\begin{align}
&\left(\mint_{\cB_{2\sigma} (\bar{x}_k (p))} f^2\right)^{\sfrac{1}{2}}
\leq\; \left(\mint_{\cB_{2\sigma} (\bar{x}_k (p))} (f-\bar{f}_{2\sigma})^2\right)^{\sfrac{1}{2}} + \sum_{i=0}^{j-1} |\bar{f}_{2^{1-i} \sigma} - \bar{f}_{2^{-i} \sigma}| + |\bar{f}_{2^{1-j} \sigma} - \bar{f}_{\bar{s} t(p)}|\nonumber\\
&\qquad\qquad\qquad\qquad\qquad +\left( \mint_{\cB_{\bar{s} t(p)} (\bar{x}_k (p))} |f-\bar{f}_{\bar{s} t(p)}|^2\right)^{\sfrac{1}{2}} + \left(\mint_{\cB_{\bar{s} t(p)} (\bar{x}_k (p))} f^2\right)^{\sfrac{1}{2}} \nonumber\\
\;\stackrel{\mathclap{\eqref{e:L2_basso}}}{\leq}
\; &C \sum_{i=0}^{j-1} \left(\frac{1}{(2^{1-i}\sigma)^{m-2}}
\int_{\cB_{2^{1-i}\sigma} (\bar{x}_k (p))} |D\bar{N}_k|^2\right)^{\sfrac{1}{2}} 
+\sqrt{\frac{\vartheta}{2}} \bh_k\, .\label{e:ancora_un_fix}
\end{align}
In the previous lines we have used repeatedly $|Df|\leq |D\bar{N}_k|$, the classical Poincar\'e inequality and the following simple 
Morrey-type estimate (which is also a consequence of the Poincar\'e inequality)
\begin{equation*}
(\bar{f}_{2t} -\bar{f}_t)^2 \leq \frac{C_0}{t^{m-2}} \int_{\cB_{2t} (\bar{x}_k (p))}
|Df|^2\, .
\end{equation*}
Note that such constant $C_0$ (and the constant for the Poincar\'e inequality) depends only upon the 
regularity of the underlying manifold $\bar\cM_k$, 
and, hence, can be assumed independent of $k$.
Summarizing, if \eqref{e:dal_basso_finale} were to fail for every radius in the interval $]\bar s t (p), 2\sigma]$, from \eqref{e:ancora_un_fix} we would conclude
\[
\mint_{\cB_{2\sigma} (\bar{x}_k (p))} f^2 \leq \bh_k^2 \vartheta \left(\frac{1}{\sqrt{2}} + C c_0\frac{1}{\sigma^{\sfrac{\alpha}{2}}}
\sum_{i=0}^{j-1} (2^{1-j}\sigma)^{\sfrac{\alpha}{2}}\right)^2 \leq \bh_k^2 \vartheta \left(\frac{1}{\sqrt{2}} + c_0 C(\alpha)\right)^2
\]}
Since $C(\alpha)$ depends on $\alpha$, $m$ and $Q$, but does not depend on $k$,
for $c_0$ chosen sufficiently small the latter inequality would contradict \eqref{e:dall'alto}.
Note that \eqref{e:inside} follows by a simple
triangular inequality.

\medskip

\textit{Step 3.}
Finally, we show that \eqref{e:dal_basso_finale} and \eqref{e:inside} lead to a
contradiction.
Consider a covering of $\Lambda_k$ with balls $\bB^i := \bB_{20 \varrho_{p_i}} (p_i)$
with the property that the corresponding balls
$\bB_{4\varrho_{p_i}} (p_i)$ are disjoint. We then can estimate
\begin{align*}
\frac{\eta}{2} & \leq C_0
\sum_i \varrho_{p_i}^{m-2+\alpha}
\stackrel{\eqref{e:dal_basso_finale}}{\leq} \frac{C_0}{c_0}\frac{\sigma^\alpha}{\vartheta\bh_k^2} \sum_i \int_{\cB_{\varrho_{p_i}} (\bar{x}_k (p_i))} |D\bar{N}_k|^2\\
&\leq \frac{C_0}{c_0}\frac{\sigma^\alpha}{\vartheta\bh_k^2} \int_{\cB_{\frac32}} |D\bar{N}_k|^2
\stackrel{\eqref{e:rev_Sob2}}{\leq}C\frac{\sigma^\alpha}{\vartheta},
\end{align*}
where $C_0>0$ is a dimensional constant.
In the last line we have used that, thanks to \eqref{e:inside}, the balls $\cB_{\varrho_{p_i}} (\p_{\bar\cM_k}(p_i))$ are pairwise disjoint and that, provided $\sigma$ is smaller
than $\frac{1}{32}$ and $k$ large enough, they are all contained in $\cB_\frac32$.
Since $\vartheta$ and $c_0$ are independent of $\sigma$, the above inequality reaches the
desired contradiction as soon as $\sigma$ is fixed sufficiently small. This will only require a sufficiently small $\bar{\ell}_k$,
which by \eqref{e:vanishing} is ensured for $k$ sufficiently large.

\section{Harmonicity of the limit}\label{s:fine!!}

In this section we prove Theorem~\ref{t:sospirato_blowup} and conclude our argument. We continue to follow the notation of
the previous section, in particular recall the maps defined in (BU1) and (BU2) of Section \ref{ss:blowup}

\subsection{First estimates} Without loss of generality we might translate the manifolds $\bar \cM_k$ so that the
rescaled points $\bar{p}_k =\bar{r}_k^{-1}\Phii_{j(k)} (0)$ coincide all with the origin.
Let $\bar F_k: \cB_\frac32 \subset\bar\cM_k \to \Iq(\R^{m+n})$ be the multiple valued map
given by $\bar{F}_k (x) := \sum_i \a{x+ (\bar{N}_k)_i (x)}$ and, to
simplify the notation, set $\p_k := \p_{\bar{\cM}_k}$.
We start by showing the existence of a 
suitable exponent $\gamma>0$ such that
\begin{gather}
\Lip (\bar{N}_k|_{\cB_{\sfrac{3}{2}}}) \leq C \bh_k^\gamma\quad\text{and}\quad \|\bar{N}_k\|_{C^0 (\cB_{\sfrac32})} \leq C (\bmo^{j(k)} \bar{r}_k)^\gamma,\label{e:Lip_riscalato}\\
\mass((\bT_{\bar{F}_k} - \bar{T}_k) \res (\p_k^{-1} (\cB_\frac32)) \leq C \bh_k^{2+2\gamma},\label{e:errori_massa_1000}\\
\int_{\cB_\frac32} |\etaa\circ \bar{N}_k| \leq C \bh_k^2\label{e:controllo_media}\, .
\end{gather}
Indeed, set $p_{j(k)} = \Phii_{j(k)} (0)$. Using the domain decomposition of Section \ref{s:regions} (note
that $\frac32 \bar r_k \in ] \frac{s_{j(k)}}{t_{j(k)}}, 3 [$)
and arguing in an analogous way we infer that
\begin{gather*}
\|N_{j(k)}\|_{C^0 (\cB_{\frac32\bar{r}_k} (p_{j(k)}))} \leq C (\bmo^{j(k)})^{\frac{1}{2m}} \bar{r}_k^{1+\beta_2}
\quad \text{and} \quad
\Lip (N_{j(k)}\vert_{\cB_{\frac32\bar{r}_k}(p_{j(k)})}) \leq C (\bmo^{j(k)})^{\gamma_2} \max_i \ell_i^{\gamma_2}
\\
\mass\big((\bT_{F_{j(k)}} - T_{j(k)}) \res \p_k^{-1} (\cB_{\frac32\bar{r}_k} (p_{j(k)}))\big) \leq \sum_i (\bmo^{j(k)})^{1+\gamma_2}
\ell_i^{m+2+\gamma_2}\label{e:errori_massa_10},\\
\int_{\cB_{\frac32\bar{r}_k} (p_{j(k)})} |\etaa\circ N_{j(k)}| \leq C \bmo^{j(k)} \bar{r}_k^{}
\sum_i \ell_i^{2+m+\sfrac{\gamma_2}{2}} + \frac{C}{\bar{r}_k^{}} \int_{\cB_{\frac32\bar{r}_k} (p_{j(k)})} |N_{j(k)}|^2\,,
\end{gather*}
where this time, for the latter inequality we have used \cite[Theorem~2.4 (2.4)]{DS4} with $a= \bar{r}_k^{}$.
On the other hand, { again by the arguments of Section \ref{s:regions} (see for instance  
\eqref{e:N_sotto}, \eqref{e:Dirichlet_sotto} and \eqref{e:nuova_referenza})} and Corollary~\ref{c:rev_Sob},
we see that
\begin{equation}\label{e:Sobolev_rovesciato}
\sum_i \bmo^{j(k)} \ell_i^{m+2+\frac{\gamma_2}{4}} \leq C_0 \int_{\cB_{\frac32\bar{r}_k} (p_{j(k)})} (|DN_{j(k)}|^2 + |N_{j(k)}|^2) \leq C \bar{r}_k^{-2} \int_{\cB_{\bar s_k} (p_{j(k)})} |N_{j(k)}|^2,
\end{equation}
from which \eqref{e:Lip_riscalato}-\eqref{e:controllo_media} follow by a simple
rescaling (the constant $C$ on the right hand side of \eqref{e:Sobolev_rovesciato} depends on $T$ but not on $k$).

It is then clear that the strong $L^2$ convergence of $N^b_k$ is a 
consequence of these bounds and of the Sobolev embedding
(cf.~\cite[Proposition~2.11]{DS1}); whereas, by \eqref{e:controllo_media},
\begin{equation*}
\int_{\cB_\frac32} |\etaa\circ N^b_\infty| =
\lim_{k\to +\infty}\int_{\cB_\frac32} |\etaa\circ N^b_k|
\leq C \lim_{k\to +\infty}\bh_k =0\, .
\end{equation*}
Finally, note that $N^b_\infty$ must take its values
in $\{0\}\times \R^{\bar{n}} \times \{0\}$.
Indeed, considering the tangential part of $\bar N_k$ given by
$\bar{N}^T_k(x) := \sum_i \a{\p_{T_x \bar \Sigma_k} (\bar{N}_k(x))_i}$,
it is simple to verify that $\cG (\bar{N}_k, \bar{N}^T_k)\leq C_0 |\bar{N}_k|^2$,
which leads to 
\[
\int_{\cB_{\sfrac32}} \cG (N^b_k, \bh_k^{-1} \bar{N}^T_k \circ \be_k)^2 \leq C_0 \bh_k^{-2} \int_{\cB_{\sfrac32}} |\bar N_k|^4 \stackrel{\eqref{e:Lip_riscalato}}{\leq} C (\bmo^{j(k)} \bar{r}_k)^{2\gamma} \to 0 \quad \text{as } k\to +\infty,
\]
and, by the convergence of $\bar\Sigma_k$ to $\R^{m+\bar n}\times\{0\}$,
gives the claim.

\subsection{A suitable trivialization of the normal bundle}
By Lemma~\ref{l:vanishing}, we can consider for every $\bar\cM_k$ an
orthonormal frame of $(T\bar\cM_k)^\perp$, $\nu^k_1, \ldots, \nu^k_{\bar{n}},
\varpi^k_1, \ldots \varpi^k_l$ with the property that
$\nu^k_j (x) \in T_x \bar\Sigma_k$,
$\varpi^k_j (x) \perp T_x \bar\Sigma_k$ and (cf.~\cite[Lemma~A.1]{DS2})
\[
\nu^k_j \to e_{m+j} \quad \text{and}\quad
\varpi^k_j \to e_{m+\bar n+ j} \quad \mbox{in $C^{2,\kappa/2}(\bar \cM_k)$ as $k\uparrow\infty$}
\]
(for every $j$: here $e_1, \ldots , e_{m+\bar{n}+l}$ is the standard basis of $\mathbb R^{m+\bar{n}+l} = \mathbb R^{m+n}$).
We next claim the existence of maps
$\psi_k : \bar\cM_k \times \R^{\bar n} \to \R^l$ converging to $0$ in
$C^{2, \kappa/2}$ (uniformly bounded in $C^{2, \kappa}$) and of $\delta>0$ (independent of $k$) such that, for every $v\in T_p\bar\cM_k$ with $|v|\leq \delta$,
\[
p + v\in \bar\Sigma_k
\Longleftrightarrow
v^\perp= \psi_k(p,v^T),
\]
with $v^T = (\langle v, \nu_1^k\rangle, \ldots, \langle v, \nu_{\bar n}^k\rangle)\in\R^{\bar n}$
and $v^\perp = (\langle v, \varpi_1^k\rangle, \ldots, \langle v, \varpi_{l}^k\rangle) \in \R^l$.
To see this, consider the map
\[
\Phi_k : \bar\cM_k \times \R^{\bar n} \times \R^l \ni (p, z, w) \mapsto 
p + z^j \nu^k_j + w^j \varpi^k_j \in \R^{m+n},
\]
where we use the Einstein convention of summation over repeated indices.
It is simple to show that the frame can be chosen so that $D\Phi_k (0,0) = \Id$ and,
hence, $\Phi_k^{-1}(\bar\Sigma_k)$ can be written locally as a graph of a function
$\psi_k$ satisfying the claim above.

Note that, by construction we also have that $\psi_k(p,0) = |D_w\psi_k(p,0)|=0$ 
for every $p \in \bar\cM_k$, which in turn implies
\begin{align}\label{e:taylorino_1}
|D_x \psi_k (x,w)|\leq C |w|^{1+\kappa}, \quad
|D_w \psi_k (x, w)|\leq C |w| \quad \text{and} \quad
|\psi_k (x,w)|\leq C |w|^2.
\end{align}
Given now any $Q$-valued map 
$u = \sum_i\a{u_i}: \bar{\cM}_k \to \Iq (\{0\}\times \R^{\bar n} \times \{0\})$
with $\|u\|_{L^\infty} \leq \delta$,
we can consider the map $\mathbf{u}_k := \psi_k (x, u)$ defined by
\[
x\mapsto \sum_i \a{ (u_i)^j \nu_j^k (x)
+ \psi_k^j(x, u_i(x)) \varpi^k_j (x)},
\]
where we set
$(u_i)^j:= \langle u_i (x), e_{m+j}\rangle$, 
$\psi_k^j(x,u_i(x)):= \langle \psi_k (x, u_i (x)), e_{m+\bar n +j}\rangle$ (again we 
use Einstein's summation convention).
Then, the differential map
$D\mathbf{u}_k := \sum_i \a{D(\mathbf{u}_k)_i}$ is given by
\begin{align*}
D(\mathbf{u}_k)_i &= D(u_i)^j  \nu_j^k + \left[D_x \psi_k^j (x, u_i) +
D_w \psi_k^j (x, u_i) Du_i\right] \varpi^k_j\\
&\qquad + (u_i)^j D \nu_j^k + \psi_k^j (x, u_i) D \varpi^k_j\, . 
\end{align*}
Taking into account that $\|D\nu_i^k\|_{C^0}+\|D\varpi^k_j\|_{C^0} \to 0$
as $k\to +\infty$, by \eqref{e:taylorino_1} we deduce that
\begin{align}
\left|\int \big(|D \mathbf{u}_k|^2 - |D u|^2\big)\right| 
\leq C \int \left(|Du|^2|u| + |Du||u|^{1+\kappa} + |u|^{2+2\kappa}\right)
+ o(1) \int \left(|u|^2+|Du|^2\right)\, .\label{e:stimozza}
\end{align}

Now we clearly have $\bar{N}_k (x)= \psi_k (x, \bar{u}_k)$
for some Lipschitz $\bar{u}_k: \bar{\cM}_k\to \Iq (\R^{\bar n})$
with $\|\bar{u}_k\|_{L^\infty} = o(1)$ by \eqref{e:Lip_riscalato}.
Setting $u^b_k:= \bar{u}_k \circ \be_k$, we conclude from \eqref{e:rev_Sob},
\eqref{e:Lip_riscalato} and \eqref{e:stimozza} that
\begin{align}\label{e:stimozza_10}
\lim_{k\to +\infty}\int_{B_{\frac32}} \left( |DN^b_k|^2 - \bh_k^{-2} |Du^b_k|^2\right) =0,
\end{align}
and $N^b_\infty$ is the limit of  $\bh_k^{-1} u^b_k$.

\subsection{Competitor function}
We now show the $\D$-minimizing property of $N^b_\infty$.
Clearly, there is nothing to prove if its Dirichlet energy vanishes.
We can therefore assume that there exists $c_0>0$ such that
\begin{equation}\label{e:reverse_control}
c_0 \bh_k^2 \leq \int_{\cB_\frac32} |D\bar{N}_k|^2\, .
\end{equation}
Assume there is a radius $t \in \left]\frac54,\frac32\right[$ and a 
function $f: B_\frac32 \to \Iq(\R^{\bar n})$ such that
\[
f\vert_{B_\frac32\setminus B_t} = N^b_\infty\vert_{B_\frac32 \setminus B_t} \quad\text{and}\quad
\D(f, B_t) \leq \D(N^b_\infty, B_t) -2\,\delta,
\]
for some $\delta>0$.
We can apply \cite[Proposition~3.5]{DS3} to the functions $\bh_k^{-1}u_k^b$
and find $r \in ]t,2[$ and competitors $v^b_k$ such that,
for $k$ large enough,
\begin{gather*}
v^b_k\vert_{\de B_r} = u^b_k\vert_{\de B_r}, \quad
\Lip (v^b_k) \leq C \bh_k^\gamma, \quad |v^b_k| \leq C (\bmo^k \bar{r}_k)^\gamma,\\
\int_{B_\frac32} |\etaa\circ v^b_k| \leq C \bh_k^2\quad\text{and}\quad
\int_{B_\frac32} |Dv^b_k|^2 \leq \int |Du^b_k|^2 - \delta \, \bh_k^2\, ,
\end{gather*}
where $C>0$ is a constant independent of $k$ and $\gamma$ the exponent of  \eqref{e:Lip_riscalato}-\eqref{e:controllo_media}.
Clearly, by Lemma~\ref{l:vanishing} and \eqref{e:taylorino_1}, 
the maps $\tilde{N}_k = \psi_k (x, v^b_k \circ \be_k^{-1})$ satisfy 
\begin{gather*}
\tilde{N}_k \equiv \bar{N}_k \quad \mbox{in $\cB_\frac32\setminus\cB_t$},\quad
\Lip (\tilde{N}_k) \leq C \bh_k^\gamma, \quad |\tilde{N}_k| \leq C (\bmo^k \bar{r}_k)^\gamma,\\
\int_{\cB_\frac32} |\etaa\circ \tilde{N}_k| \leq C \bh_k^2
\quad\text{and}\quad
\int_{\cB_\frac32} |D\tilde{N}_k|^2 \leq \int_{\cB_\frac32} |D\bar{N}_k|^2 - \delta \bh_k^2.
\end{gather*}

\subsection{Competitor current} Consider finally the map
$\tilde{F}_k (x) = \sum_i \llbracket x+\tilde{N}_i (x)\rrbracket$. The current $\bT_{\tilde{F}_k}$ coincides with 
$\bT_{\bar{F}_k}$ on $\p_k^{-1}( \cB_\frac32\setminus \cB_t)$.
Define the function $\varphi_k(p) = \dist_{\bar{\cM}_k} (0, \p_k (p))$
and consider for each $s\in \left]t,\frac32\right[$ the slices 
$\langle \bT_{\tilde{F}_k} - \bar{T}_k, \varphi_k, s\rangle$.
By \eqref{e:errori_massa_1000} we have
\[
\int_t^\frac32 \mass (\langle \bT_{\tilde{F}_k} - \bar{T}_k, \varphi_k, s\rangle) \leq C \bh_k^{2+\gamma}\, .
\]
Thus we can find for each $k$ a radius $\sigma_k\in \left]t, \frac32\right[$ on which 
$\mass (\langle \bT_{\tilde{F}_k} - \bar{T}_k, \varphi_k, \sigma_k\rangle) \leq C \bh_k^{2+\gamma}$. By the isoperimetric
inequality (see \cite[Remark 4.3]{DS3}) there is a current $S_k$ such that
\[
\partial S_k = 
\langle \bT_{\tilde{F}_k} - \bar{T}_k, \varphi_k, \sigma_k\rangle,
\quad \mass (S_k) \leq C \bh_k^{(2+\gamma)m/(m-1)}
\quad\text{and}\quad\supp (S_k) \subset \bar{\Sigma}_k.
\]
Our competitor current is, then, given by
\[
Z_k := \bar{T}_k\res (\p_k^{-1} (\bar\cM_k\setminus \cB_{\sigma_k})) + S_k + \bT_{\tilde{F}_k}\res
(\p_k^{-1} (\cB_{\sigma_k})).
\]
Note that $Z_k$ is supported in $\bar{\Sigma}_k$ and has the same boundary as 
$\bar{T}_k$. On the other hand, by \eqref{e:errori_massa_1000} and the bound on $\mass (S_k)$, we have
\begin{equation}\label{e:perdita_massa}
\mass (\tilde{T}_k) - \mass (\bar{T}_k) \leq \mass (\bT_{\bar F_k}) - \mass (\bT_{\tilde{F}_k}) + C \bh_k^{2+2\gamma}\, .
\end{equation}
Denote by $A_k$ and by $H_k$ respectively the second fundamental forms and mean curvatures of the
manifolds $\bar\cM_k$. Using the Taylor expansion
of \cite[Theorem~3.2]{DS2}, we achieve
\begin{align}
 \mass (\tilde{T}_k) - \mass (\bar{T}_k) &\leq \frac{1}{2} \int_{\cB_\rho} \left(|D\tilde{N}_k|^2 - |D\bar{N}_k|^2\right) 
+ C \|H_k\|_{C^0} \int \left(|\etaa\circ \bar{N}_k| + |\etaa\circ \tilde{N}_k|\right)\nonumber\\
&\qquad + \|A_k\|_{C^0}^2 \int \left( |\bar{N}_k|^2 + |\tilde{N}_k|^2\right) + o (\bh_k^2)
\leq -\frac{\delta}{2} \bh_k^2 + o (\bh_k^2)\, ,\label{e:fine!!}
\end{align}
where in the last inequality we have taken into account Lemma \ref{l:vanishing}.
Clearly, \eqref{e:fine!!} and \eqref{e:perdita_massa}
contradict the minimizing property of $\bar{T}_k$ for $k$ large enough and concludes
the proof.

\appendix

\section{Some technical Lemmas}\label{a:tecnici}
The following is a special case of Allard's $\eps$-regularity theory (see \cite[Chapter 5]{Sim}).

\begin{theorem}\label{t:strong_degiorgi}
Assume $T$ is area minimizing, $x\in\rD_Q (T)$ and $\|T\| ((\supp (T) \cap U)\setminus D_Q) = 0$ in some neighborhood $U$ of $x$. Then, $x\in \reg (T)$. In particular, $\rD_1 (T)\subset \reg (T)$.
\end{theorem}

\begin{proof} By simple considerations on the density, 
the tangent cones at $x$ must necessarily be all $m$-dimensional planes with
multiplicity $Q$.
This allows to apply Allard's theorem and conclude that, in a neighborhood of $x$, 
$\supp (T)$ is necessarily the graph of a $C^{1, \kappa_0}$ function for some $\kappa_0>0$.
Let $u: \R^m \to \R^{\bar{n}+l}$ be the corresponding function and $\Psi: \R^{m+\bar{n}}\to \R^l$
a $C^{3, \eps_0}$ function whose graph describes $\Sigma$. Let $\bar{u}$ consist of the first $\bar{n}$
coordinates functions of $u$. We then have that $\bar{u}$ minimizes an elliptic functional of the
form $\int \Phi (x, \bar{u} (x), D\bar{u} (x))\, dx$ where $(x,v,p)\mapsto \Phi (x,v,p)$ and $(x,v,p)\to D_p\Phi
(x,v,p)$ are of class $C^{2, \eps_0}$. We can then apply the classical regularity theory to conclude
that $\bar{u}\in C^{3, \eps_0}$ (see, for instance, \cite[Theorem~9.2]{Morrey}),
thereby concluding
that $x$ belongs to $\reg (T)$ according to Definition \ref{d:reg_sing}.  Fix next {\em any}
$x\in D_1 (T)$. By the upper semicontinuity of the density $\Theta$ (cf.~\cite{Sim}),
$\Theta \leq \frac{3}{2}$
in a neighborhood $U$ of $x$, which implies $\|T\| ((\supp (T) \cap U)\setminus D_1) = 0$.
\end{proof}

Next, we prove the following technical lemma.

\begin{lemma}\label{l:separate}
Let $T$ be an integer rectifiable
current of dimension $m$ in $\R^{m+n}$ with locally finite mass and $U$ an open set such that $\cH^{m-1} (\partial U \cap \supp (T)) = 0$ and $(\partial T) \res U = 0$. Then $\partial (T\res U)=0$.
\end{lemma}

\begin{proof} Consider $V\subset \subset \R^{m+n}$.
By the slicing Theorem \cite[4.2.1]{Fed} applied to ${\rm dist} (\cdot, \partial U)$ we conclude that 
$S_r:= T\res (V \cap U \cap \{{\rm dist}\, (x, \partial U) >r\})$ is
a normal current in ${\bf N}_m (V)$ for a.e.~$r$.
Since $\mass (T\res(V\cap U)-S_r)\to 0$ as $r\downarrow 0$, we conclude that $T\res (U\cap V)$
is in the $\mass$ closure of ${\bf N}_m (V)$. Thus, by \cite[4.1.17]{Fed}, $T\res U$ is a flat chain in $\R^{m+n}$. By \cite[4.1.12]{Fed},
$\partial (T\res U)$ is also a flat chain. It is easy to check that $\supp (\partial (T\res U)) \subset \partial U\cap \supp (T)$. Thus we can apply
\cite[Theorem 4.1.20]{Fed} to conclude that $\partial (T\res U)=0$.
\end{proof}

Recall the following theorem (for the proof see \cite[Theorem 35.3]{Sim}).

\begin{theorem}\label{t:strat}
If $T$ is an integer rectifiable area minimizing current in $\Sigma$, then
\[
\cH^{m-3+\alpha}_\infty \Big(
\supp (T)\setminus \Big(\supp (\partial T)\cup \bigcup_{Q\in\N} \rD_Q(T)\Big)\Big) = 0 \quad \forall \; \alpha>0.
\]
\end{theorem}

We finally prove the following result (first proved by Allard in an unpublished note and hence
reported in \cite{Alm}).

\begin{proposition}\label{p:exp}
Set $\pi:= \R^m\times \{0\}\subset \R^{m+n}$ and
let $\cM$ be the graph of a $C^{3,\kappa}$ function $\phii :\pi \supset B_3 (0)\to \R^m$, with $\phii (0)=0$.
Then the exponential map $\exp: B_3 (0)\to \cM$ belongs to the class $C^{2, \kappa}$.  Moreover,
if $\|\phii\|_{C^{3,\kappa}}$ is sufficiently small, then the set $\p_\pi (\exp (B_r (0)))\subset \pi$ is (for
all $r<3$) a convex set
and the maximal curvature of its boundary is less than $\frac{2}{r}$.
\end{proposition}

\begin{proof} Consider any $C^{3,\kappa}$ chart $x: \cM\to \Omega$, for instance the one induced by the
graphical structure. It is then obvious that the components $g_{ij}$ of the Riemannian metric (induced by the Euclidean
ambient space on the submanifold $\cM$) are $C^{2,\kappa}$. We let $\nabla$ be the Levi-Civita connection on $\cM$ for which $g$ is parallel
and consider the corresponding Christoffel symbols $\Gamma_{jk}^i$ (in the fixed coordinate patch). Using the standard formula which expresses the Christoffel
symbols $\Gamma_{jk}^i$ in terms of the metric $g_{sr}$ (see for instance \cite[Proposition 2.54]{GHL}), it is easy to see that the former are $C^{1,\kappa}$. The careful reader will notice that these objects are usually defined in standard textbooks assuming that the metric
is $C^\infty$, but in order to have a unique Levi-Civita connection it is enough that the metric is $C^1$, see the proof of
\cite[Theorem 2.51]{GHL} (in fact the Levi-Civita connection on $\cM$ can also be recovered by differentiating in the Euclidean ambient space and projecting the result onto the tangent space to $\cM$). Similarly, for $C^2$ metrics we can use the intrinsic definition of the Riemann curvature tensor
as in \cite[Definition 3.3]{GHL}. From the standard formula in \cite[3.16]{GHL} we easily conclude that the components of
this tensor are $C^{0, \kappa}$. However, by \cite[Lemma A.1]{DS1} we can choose a $C^{2, \kappa}$ orthonormal frame $\nu_1, \ldots , \nu_n: \Omega \to \mathbb R^{m+n}$ for the normal bundle of $\cM$ and the curvature tensor can be computed via the Gauss' equations as in \cite[5.8b)]{GHL}: we thus conclude that the components of the Riemann tensor are 
in fact $C^{1,\kappa}$. Again, although the references above carry on all computations in the $C^\infty$ setting, it can be easily
checked that these work in a straightforward way under our regularity assumptions.

Let next $\Phi (t,v):= \exp (vt)$ (the fact that the exponential map is well defined will be justified in few lines). Fix a $C^{3,\kappa}$ coordinate patch on $\cM$ where $0$ is the
origin,  using the graphical structure of $\cM$ over $T_0 \cM$. Set $t\mapsto \gamma (t)= \Phi (t,v)$ and use the notation
$\gamma_j'$ for the components of $\gamma'$ in the fixed chart $x:\Omega\to \cM$ (so, $\gamma' (t) = \sum_j \gamma'_j (t) \frac{\partial}{\partial x_j}$). $\gamma$ satisfies the system of differential equations
\[
\gamma_j''(t) + \sum_{ik} \Gamma_{ik}^j (\gamma (t)) \gamma'_i (t) \gamma'_k (t) = 0,
\]
with the initial conditions $\gamma (0) =0$ and $\gamma' (0) = v$, cf. \cite[Definition 2.77]{GHL}. It follows thus that the maps $\Phi$ and $\partial_t \Phi$ are $C^{1,\kappa}$; incidentally, this shows that the exponential map is well-defined (in fact, standard textbooks on ODEs only provide $C^1$ regularity; however the usual proof of $C^1$ regularity via Gronwall Lemma on the linearized ODEs for the derivative $\partial_v \Phi$ can be easily modified to prove $\partial_w \Phi \in C^{0,\alpha}$; cf. \cite[Section 9]{Amann}). 

Fix now a tangent vector $e$ at $0$, a point $p = \exp (v) \in \cM$  and perform a parallel transport of 
$e$ along the (``radial'') geodesic segment $[0,1]\ni t\mapsto \exp (tv)$ to define $e(p)$. We claim that the corresponding vector
field is $C^{1,\kappa}$. Indeed, fix any orthonormal tangent frame $f_1, \ldots f_m$ which is $C^{2, \kappa}$. 
Let
\[
e (\exp (tv)) = \sum_i \alpha_{v,i} (t) f_i (\Phi (t,v)) = \sum_{i,k} \alpha_{v,i} (t) \sum_k \varphi_{ik} (\Phi (t,v)) \frac{\partial}{\partial x_k}\, ,
\]
where the functions $\varphi_{ik}$ are $C^{2,\kappa}$. Recall that the a vector field $X (t) = \sum_j X_j (t) \frac{\partial}{\partial x_j}$ along a $C^1$ curve $c$ with tangent $c' (t) = \sum_i c_i' (t) \frac{\partial}{\partial x_i}$ is parallel if and only if
\[
X_i ' (t) = - \sum_{j,k} \Gamma_{jk}^i (c(t)) c'_j (t) X_k (t)\, ,
\]
cf. \cite[Theorem 2.68 and equation (2.69)]{GHL}.
We therefore conclude that the coefficients $\alpha_{v,i} (t)$ must satisfy a system of ODEs of the form
\[
\alpha'_{v,i} (t) = - \sum_j \alpha_{v,j} (t) F_{ij} (\Phi (t,v), \partial_t \Phi (t,v)) 
\]
where $(t,v)\mapsto
F_{ij} (\Phi (t,v), \partial_t \Phi (t,v))$ are $C^{1,\kappa}$ maps. Thus the existence of $e$ and
the claimed regularity of $(t,v)\mapsto \alpha_{v,i} (t)$
follow from the standard theory of
ODEs.

Recall also that the parallel transport keeps the angle between vectors constant, cf. \cite[Proposition 2.74]{GHL}.
We conclude that there exists an orthonormal frame $e_1, \ldots, e_m$ of class $C^{1, \kappa}$ which is parallel along geodesic rays emanating
from the origin.
Next, consider the map $(w,v,t) \mapsto \partial_w \Phi (t,v)$ where $w$ varies in $\R^m$.
Fix $w$ and $v$ and consider again the curve $\gamma (t)$
above and the vector $\eta_{v,w} (t) = \partial_w \Phi (t,v)$. We claim that $\eta$ satisfies the Jacobi equation along the geodesic $\gamma$, with initial data $\eta_{v,w} (0) = 0$ and $\eta_{v,w}' (0) = w$. 
More precisely, if we write the vector field in the frame $e_i$ as $\eta (t) = \sum_i \eta_i (t) e_i (\gamma (t))$, the Jacobi equation is
\begin{equation}\label{e:Jacobi}
\eta_{v,w,i}'' (t) = - \sum_j R_{\gamma (t)} (e_j (\gamma (t)), \gamma' (t), \gamma' (t), e_i (\gamma (t))) \eta_{v,w,j} (t)\, ,
\end{equation}
where $R$ depends on the Riemann tensor (cf. \cite[Theorem 3.43]{GHL}). Note that we do not have the usual smoothness assumptions under which \eqref{e:Jacobi} is derived in standard textbooks. We can however proceed by regularizing our manifold $\cM$ via convolution of the function of which the manifold is a graph. We fix the corresponding graphical charts for the regularized manifolds and observe that the exponential maps in these coordinates have uniform $C^{1,\kappa}$ bounds from the corresponding ODEs and thus will converge to $\Phi$ in $C^1$. Similarly one concludes the obvious convergence statements for the Riemann tensor and thus the right hand side of \eqref{e:Jacobi} for the corresponding objects converge uniformly. This justifies, in the limit, that $\eta_{v,w,i}$ is twice differentiable (in time) and that \eqref{e:Jacobi} holds.

Taking into account that $\gamma (t) = \Phi (t,v)$ and $\gamma'(t)=
\partial_t \Phi (t,v)$ we conclude that $\eta_{v,w,i}$ satisfies an ODE of the type $\eta_{v,w,i}'' (t) = \Lambda (t,v, \eta_{v,w,i} (t))$ where the function
$\Lambda$ is $C^{1,\kappa}$ in all its entries. We thus conclude that the map $(v,w,t)\mapsto \eta_{v,w} (t) = \partial_w \Phi (t,v)$ is a $C^{1,\kappa}$ map. Since
$d\exp (v) (w) =  \partial_w \Phi (v,1)$, this implies that the exponential map is $C^{2,\kappa}$.

As for the last assertion, for $\|\phii\|_{C^{3,\kappa}}$ sufficiently small we conclude from the discussion above
that $\p_\pi\circ \exp$ is $C^2$ close to the identity, which implies the desired statement.
\end{proof}

\bibliographystyle{plain}
\bibliography{references-BU}

\end{document}